\documentclass[12pt]{article}
\usepackage[T1]{fontenc}
\usepackage[utf8]{inputenc}
\usepackage{lmodern}
\usepackage{microtype}
\usepackage{amsthm}
\usepackage[titletoc, title]{appendix}
\usepackage{indentfirst}
\usepackage{subfigure}
\usepackage{enumitem}
\usepackage{amssymb}
\usepackage{amsmath,amsthm,marvosym,wasysym,mathdots}
\usepackage[natbibapa,nodoi]{apacite}
\usepackage{color}
\usepackage{setspace}
\usepackage{hyperref}
\usepackage{mathrsfs}
\usepackage{amsmath}
\usepackage{graphicx}
\usepackage{ae,aecompl}
\usepackage{tikz,pgfplots}
\usepackage{pgf-pie}
\usepackage[multiple]{footmisc}
\usepackage{mathtools}                      
\usepackage{commath}
\mathtoolsset{showonlyrefs=true}            
\usepackage[normalem]{ulem}

\hypersetup{%
  colorlinks=false,
  pdfborder = {0 0 0.5 [3 0]}
}
\usepackage{breakurl}

\usepackage{multirow}
\usepackage[left=1in,right=1in,top=1in,bottom=1in]{geometry}

\pgfplotsset{%
    every tick label/.append style={
        font=\tiny,
        /pgf/number format/.cd,
        fixed,
        precision=2,
        /tikz/.cd
    },%
    legend style={font=\small},
    scaled ticks=false,%
}
\usetikzlibrary{external}
\usepackage{shellesc}
\immediate\write18{mkdir tikz}

\usepackage[margin=1cm]{caption}

\newcommand\la{\lambda}
\newtheorem{theorem}{Theorem}
\newtheorem*{theorem*}{Theorem}

\newtheorem{lemma}[theorem]{Lemma}

\newtheorem{proposition}[theorem]{Proposition}
\newtheorem{remark}[theorem]{Remark}

\def\R{{\mathbb R}} 
\newcommand{\EE}{{\mathord{I\kern -.33em E}}}
\def\E{{\mathbb E}} 

\def\1{1{\hskip -3.3 pt}\hbox{I}}

\providecommand{\varitem}{}


\numberwithin{equation}{section}
\numberwithin{theorem}{section}

\newcommand{\Rnneg}{\R_{\geq 0}}

\newcommand{\rewardr}{\lambda}
\newcommand{\meanhashr}{\bar{\alpha}}
\newcommand{\opthashr}{\alpha^*}
\newcommand{\eqmeanhashr}{\expandafter\bar\opthashr}
\newcommand{\costr}[1]{c#1}

\newcommand{\lop}{\mathcal{L}}

\newcommand{\K}{\mathcal{K}}

\begin{document}

\title{A Mean Field Games Model for Cryptocurrency Mining}

\author{
Zongxi Li%
\thanks{ORFE Department, Princeton University, Princeton, USA.}
\and
A.\ Max Reppen%
\thanks{Questrom School of Business
 Boston University,
 Boston, MA, 02215, USA.
 Partly supported by the Swiss National Science Foundation grant SNF 181815.}
\and
Ronnie Sircar%
\footnotemark[1]
}

\date{December 2019; revised January 14, 2022}

\maketitle

\begin{abstract}
 We propose a mean field game model to study the question of how centralization of reward and computational power occur in Bitcoin-like cryptocurrencies.
    Miners compete against each other for mining rewards by increasing their computational power.
    This leads to a novel mean field game of jump intensity control, which we solve explicitly for miners maximizing exponential utility, and handle numerically in the case of miners with power utilities.
   We show that the heterogeneity of their initial wealth distribution leads to greater imbalance of the reward distribution, and increased wealth heterogeneity over time, 
   or a ``rich get richer'' effect.
    This concentration phenomenon is aggravated by a higher bitcoin mining reward, and reduced by competition.
    Additionally, an advantaged miner with cost advantages such as access to cheaper electricity, contributes a significant amount of computational power in equilibrium, unaffected by competition from less efficient miners.
    Hence, cost efficiency can also result in the type of centralization seen among miners of cryptocurrencies.
\end{abstract}

\section{Introduction}
\label{Bitcoin mining}
Blockchain technologies solve the elusive problem of creating a decentralized ledger.
Proponents have made the case that a payment and banking system founded on blockchain technologies could conceivably allay data privacy concerns, and provide wider access to financial services.
Various forms of digital currency solutions have been developed in the past, but Bitcoin is the famous realization that finally created a secure distributed ledger (see \cite{Nakamoto2008}).
Since its creation in January 2009, its fame and popularity have grown rapidly.
The supply of bitcoins is constantly growing, but artificially limited to $21$ million, of which almost $19$ million are in circulation now. 

In a cryptocurrency network following a so-called proof-of-work protocol, independent ``miners'' compete for the right to record the next transaction block on the blockchain.
They do so by solving computational puzzles: once a miner obtains a solution, the corresponding block is added on top of the blockchain, and the miner obtains a reward.
The computational puzzle is designed such that there is no known better way of solving it than brute force calculation. 
In other words, the chance of getting the reward is proportional to the computational power or the hash rates that miners  contribute. 
The purpose of the puzzle is to disincentivize bad behavior by forcing miners to provide proof-of-work that some effort has been exerted -- electricity paid. 
Moreover, the difficulty of the puzzle varies to maintain a consistent solving time, for example 10 minutes.

The empirical analysis in \cite{Kondor2014} shows that the accumulation of bitcoins tends to occur among a small number of miners, which suggests centralization in the market.
This raises the following questions: 
What incentives drive miners' competitive behavior?
How does the centralization of the reward occur in a decentralized mining environment?
What factors impact this centralization? 

Another emerging phenomenon is major miners who have access to cheaper electricity or more efficient hardware and thus are cost-advantaged.
Bitmain controls AntPool and BTC.com and accounts for around $33\%$ of the total hash rate in the world as of June 2019 (see Figure \ref{fig: Bitcoin hash rate distribution.}).
This number includes Bitmain's computational power as well as that of miners who join the pools.
By rough estimates, Bitmain accounted for about $4.5\%$ of all mining power for Bitcoin mining in October 2018.\footnote{\url{https://web.archive.org/web/20181017133438/https://blog.bitmain.com/en/hashrate-disclosure/}}\footnote{\url{https://btc.com/stats/diff}}
Moreover, as \cite{Taylor2017} points out, some mining entities also develop application-specific integrated circuits (ASICs) and create related data centers with low energy cost.
One may ask: What leads to this centralization of computational power? What advantages do some miners have?

We are interested in understanding and modeling the interaction between cryptocurrency miners, the consequent evolution of wealth inequality among them, and the potential centralization of mining power.
In the context of cryptocurrency mining, an economic model captures the cost-reward structure of mining: the cost -- typically the marginal cost of electricity, and the potential reward -- the units of currency rewarded to successful miners.

An ongoing regulatory issue is as to whether cryptocurrencies are currencies or commodities (or even share-like assets offered in initial coin offerings).
The Commodity Futures Trading Commission (CFTC) in the US classifies them as commodities, and their electronic structure of production mirrors the uncertainty and language of mining exhaustible resources. The latter connects us to game theoretic models that have been developed to try and explain energy production from various sources many of which, like oil, are in finite supply (see, for instance, \cite{fields_survey}). 

Similarly to models of natural resource extraction, the miners are producers, but the product is numbers.
These numbers, called hashes, come from processing transaction data.
Because of the demand for processing, there is effectively also a demand for hashes.
So, just like in any other market, the suppliers (miners) are competing to fill this demand.
And as in other markets, the (expected) reward per hash obtained is diminishing in the aggregate supply (mining power).
However, calling it a market for hashes is masking the actual mechanism, because it is not the numbers that are valuable.
Instead, the value lies in the process that `extracts' the numbers (mining), and the numbers themselves are just a byproduct. 
Therefore the basis of the models introduced in this paper is the probabilistic relation between competitive hashing and cryptocurrency discovery.

Our paper quantifies the competition between miners by adopting a tractable mean field games approach.
The idea behind mean field games is that, with many participants, any particular player has little impact on any other player when they interact through the mean of their actions---hashing in our case.
As the number of players grows,
one can first view an individual's decision making problem as being against a mean field competition, knowing that their individual contribution to the mean field is infinitesimally small.
The final step is a fixed point condition in which the mean optimal control should coincide with the aforementioned mean field.
This approach leads to a computationally tractable model, which is
intuitively a good approximation to the finite player game, precisely because each player's impact dissipates in the mean of many players.

\subsection{Related literature}
Our work is related to a growing literature on cryptocurrencies.
A game-theoretic model is developed in \cite{Easley2019} to show the emergence of transaction fees in the Bitcoin payment system. \cite{Brunnermeier2018} point out the blockchain trilemma, and analyze when decentralized record-keeping is economically beneficial. \cite{Xiong2018} explore a model to study initial coin offerings for new decentralized digital platforms. \cite{He2019} argue that the blockchain facilitates the creation of smart contracts, which can sustain market equilibria with a larger range of economic outcomes. \cite{Biais2019} use a stochastic game to show that the proof-of-work protocol results in multiple equilibria, some of which can lead to persistent divergence between chains.
A revenue management problem in the context of bitcoin selling is studied in \cite{hotelling_to_nakamoto}.
Our work differs from these studies in that we analyze centralization of both the reward and computational power in mining activities
as well as how the reward size and competition impact it. 

Our work is most closely related to recent literature on miners' strategic behavior and the centralization of mining. \cite{NBERw25592} examine mining pools' fee-setting decisions, and their unexpected impact on risk sharing and concentration. \cite{Weinberg2018} consider asymmetric costs among miners and show that lower cost leads to higher market share. On the other hand, \cite{Capponi2019} explore a two-stage mining game consisting of research and development and then competition. They explain how the arms race leads to asymmetric costs and mining centralization.
\cite{garratt2020fixed} study the importance of the cost structure in cryptocurrency mining, and they find that fixed costs improves the resilience to shocks due to its impact on exit and entry.

In contrast to the static games in these papers, our work considers continuous time mean field games, incorporating dynamically evolving wealth and mining decisions.
We refer to \cite{gueant2011mean} for an early introductory exposition on mean field games.
\cite{bertucci2020mean} provide another example of mean field game modeling of cryptocurrency mining.
Their model complements ours, as they focus on the blockchain's resilience to outside attacks, whereas ours focuses on centralization and inside attacks.
Their modeling decisions and approach are therefore very different, in order to capture the phenomena of importance in that setting.

Our work also contributes to the literature on intensity control of jump processes. It is used in the model for exploration of natural resources. \cite{Pliska1980} and \cite{Arrow1982} study the optimal consumption rule of a natural resource. They use a point process to model the uncertainty of the discoveries for new sources of supply, where the control is exploration effort. Later on, \cite{Soner1985} considered a similar model with holding cost, and established the existence and uniqueness of solution to the Bellman equation. Intensity control models are also used in revenue management and dynamic pricing. A buffer flow system with jumps is considered by \cite{Li1988}, where the cumulative production and demand are modeled by two counting processes, with intensity controlled by production capacity and price. In addition, \cite{Gallego1994,Gallego1997} model dynamic pricing for inventories of products. The demand for those is modeled as point processes and the intensities are controlled by setting prices. In our work, the jump process is used to represent the acquisition of the reward. The miners control the jump intensity through adjusting their computational power or hash rates. This model approach is natural due to the two important properties of Bitcoin payment system mentioned before. 

There has been recent work on games of intensity control.
For instance, \cite{Sircar2012} consider the effects of stochastic resource exploration in dynamic Cournot game, where an exhaustible producer and a green producer set the production to affect the price. \cite{Gallego2014} study dynamic pricing in an oligopolistic market. Each firm competes to sell its product and the equilibrium strategies and prices are resolved. In a mean field game setting, \cite{Chan2017} examine the impact of oil discovery, concluding that higher reserves lead to lower exploration. There the players' interaction was through producers' oil extraction rates.
In this paper, different from most works in the literature, the mean field interaction is through the players' intensities, or hash rates.

The proof-of-work system supporting Bitcoin, Ethereum,\footnote{Ethereum is working towards replacing proof-of-work with proof-of-stake in Ethereum 2.0 (Eth2).  The first step towards this migration (Phase 0) was deployed in December 2020.} and the majority of alt-coins has received heavy criticism for the high energy consumption of its miners.
The most prominent alternative consensus method is called proof-of-stake.\footnote{Proof-of-stake was first introduced in 2011 by user QuantumMechanic on the bitcointalk forums: \url{https://bitcointalk.org/index.php?topic=27787.0}}
In this system, instead of spending computing power as a requirement for creating valid blocks, the participants are instead ``randomly'' chosen in proportion to their current stake in the system.
A user holding 1\% of all coins will in the long run create 1\% of new blocks.
Because block creation is wasteless, this leads to the so-called nothing-at-stake problem in which deviation from ``good'' behavior is not punished.
We refer to \cite{brown2019formal} for a more detailed account of proof-of-stake and its drawbacks.
\cite{fanti2019compounding} and \cite{rosu2019evolution} both study the impact of rewards on the wealth distribution of participants by using the martingale property of participants' \emph{share} of assets in a proof-of-stake system.
\cite{rosu2019evolution} show that under a constant reward scheme, the limiting distribution is stable in terms of the share of total coins, and thus fair (not exhibiting a rich get richer effect).%
\footnote{%
    Although this is a very interesting result and good news for proof-of-stake systems,
    we perceive this as strongly connected to the martingale property of such systems, and believe that it does not translate to proof-of-work for two main reasons:
    First, central to the martingale property is the dilution effect (inflation tax) of introducing new coins.
    In a proof-of-stake system, this means that whereas a wealthier participant receives proportionally larger rewards, the inflation tax is also proportionally larger.
    In contrast, in a proof-of-work system, a miner does not need to hold the cryptoasset to reap the rewards, creating the possibility for a miner to avoid the inflation tax if so desired.
    Second, even though some miners could be competitive long-term, the presence of wealth considerations and risk aversion could prematurely force them out of the game against their will.
    This latter effect is apparent in the model studied here.
}
Additionally, proof-of-stake is believed to not only have the advantage of reducing energy consumption and improving fairness, but also of increasing security, as any attacker must have a stake in the system.
At time of writing, various proof-of-stake systems are under heavy development, both theory and in practice.

\subsection{Contribution and intuition}\label{our_contrib}
We introduce a dynamic competitive mining game in the presence of risk aversion and liquidity constraints.
The miners compete by exerting computational effort in an attempt to obtain the mining rewards, which are distributed based on the computational effort relative to the population aggregate.
They are expected utility maximizers, balancing the cost of computation (e.g.\ electricity) and the reward associated with block creation.

We find the equilibrium explicitly under exponential (CARA) utility (Section \ref{sec:liquid}), thereby establishing the existence and uniqueness of a solution in this case. The explicit solution allows us to understand the influence of model parameters on the equilibrium.

For miners with power (CRRA) utility, 
we find the equilibrium numerically (Section \ref{sec:illiquid}).
In this case, heterogeneity of the initial wealth distribution among miners results in preferential attachment, i.e.\ increasing heterogeneity of wealth over time.
In other words, a miner with greater wealth contributes a larger hash rate, and thus has a higher probability of receiving the next reward, whereas some miners with lesser initial wealth become disincentivized over time from participating entirely, leading to increased wealth inequality. 
Moreover, our results show that increasing centralization in mining power arises, an antithesis to the principles behind cryptocurrency security.

In an extended model (Section~\ref{sec:cost_adv}), we consider an \emph{advanced miner} with cost advantages, for instance due to having access to cheaper electricity or advanced equipment.
We show that the advanced miner accounts for a significant share of the total hash rate.
Hence, cost efficiency is another factor leading to the centralization of mining power.

The main driver of our results is that, {\em ceteris paribus}, the individual miner's ratio of expected reward and the standard deviation of the reward is decreasing in the hash rate (at least to a point).
Coupled with risk aversion, this creates an advantage to a miner willing to take on more risk by mining at a higher rate.
Although assumptions on cost structure, rewards, or preferences do not change this fundamental feature, they can modulate its strength.

The paper's technical and methodological contributions are as follows.
As each miner's action affects the other miners equally, the problem exhibits a natural symmetry that is well suited to be modeled as a mean field game.
Our adaptation of mean field games technology to this problem is novel:
to our knowledge, ours is the first mean field game of control in which the control and the aggregate (rather than the mean) of other players' control influence the intensity of a jump process.
This enables us to capture dynamic features of cryptocurrency mining competition and utility optimization in a numerically tractable way.
Finally, we also provide one of the few explicit mean field game equilbria outside of a linear quadratic framework.

\section{Model and methodology}\label{sec:model}

In this section we develop our model motivated by the specifics of proof-of-work mining and set up the equations to be tackled analytically and numerically in Section~\ref{sec:results}.
Section~\ref{sec:pow_mining} briefly presents the purpose and mechanics of cryptocurrency mining, followed by Section~\ref{sec:continuum_approximation} which describes our method for representing participant miners as a continuum.
Sections~\ref{general_structure}, \ref{oneminerprob}, \ref{sec:equilibrium} collectively describe the mathematical model for the mining game, and finally Section~\ref{sec:illiq_numerical} presents the numerical method we implement.

\subsection{Proof-of-work mining}\label{sec:pow_mining}
Under a proof-of-work protocol, mining is the process by which a block, i.e., a list of transactions, is appended to the cryptocurrency ledger, 
and by which the system controls who may choose the next block, and thus also the next transactions to be registered.
Due to the pseudonymity inherent in the system, any real-world individual or entity can trivially pose as multiple separate users, and for this reason it is problematic to, for instance, take turns in having the right to append the next block.

A solution to this is to give users the right to append the next block in proportion to their computing power by means of a computational mining game.
The underlying assumption is that computational power cannot be monopolized.
Agents are incentivized to participate by rewards to successful miners.
This is accomplished by mining puzzles that are unique to the data being processed.
This data, among other things, includes a reference to the last mined block (thus creating the `chain' of blocks) and the transactions to be processed.
When the mining puzzle is solved by one miner, the miner receives the reward, the data processed by that miner is appended to the blockchain, and the game starts over for everyone, now with a reference to the new block.

The puzzle is characterized by a binary function $h$ with range $[0,1]$ and a $\texttt{target} \in (0,1)$, the latter of which is changing over time.
A miner successfully mines a block characterized by its $\texttt{data}$ if they find an input $y$ such that $h(\texttt{data}, y) < \texttt{target}$.
Because the second input to the function $h$ is arbitrary and does not carry any information about the transactions themselves, it is referred to as a nonce.

To make this game computationally difficult, the function $h$ must have the property that there is no better way to find a nonce $y$ satisfying $h(\texttt{data}, y) < \texttt{target}$ than brute force, i.e., haphazardly evaluating the function for different nonce values.
This property is provided by so-called (cryptographic) hash functions.
The function $h$ is thus referred to as the hash function and its output as a hash.
The precise hash function varies between cryptocurrencies and the target is dynamically adjusted so that the average rate at which the population as a whole solves the puzzles is stable.

All participating miners simultaneously try one nonce after the other until one miner is lucky enough to produce a small enough hash output, below the target.
As a consequence, the probability of miner $i$ appending the next block is
\begin{equation}\label{eqn:relative_hash_rate}\footnotesize
    \frac{\text{miner $i$'s hash rate}}{\text{total hash rate}}.
\end{equation}
After the successful mining of a block, the game repeats itself in the hunt for the subsequent block.
We see from this expression that, provided no miner's hash rate is large relative to the population total, no miner will have undue power over which transactions are included in the ledger.
If this is not satisfied and one or more miners run dominant shares of mining computations, the system is \emph{centralized} and those miners have undue power over the ledger and what transactions are included.
This could allow dominant miners to conduct censorship or other forms of malicious behavior.
For more in-depth detail on the mining game, problems with centralization, and potential attacks, we refer the interested reader to \cite[Chapter 5]{Narayanan2016}.

Finally, note that because only the data changes between blocks, the process describing the arrival of solutions to the puzzle is memoryless.
This is an important property for the mathematical model in Section~\ref{sec:model}.

\subsection{Continuum mean field approximation}
\label{sec:continuum_approximation}
To study the centralization of mining power and rewards, we first recognize the mean field structure inherent in mining competition and then introduce a continuum approximation, which leads to the formulation of a continuum mean field game.

From \eqref{eqn:relative_hash_rate} we see that each player depends on the other players through the sum of their hash rates, which is proportional to the mean hash rate.
Hence, formalizing cryptocurrency mining as a mean field game is completely natural, because each miner is affected by each other miner equally (through their aggregate or mean).%
\footnote{This is in contrast to some other proposed applications of mean field games, such as banking, where the assumed symmetry in agent interaction does not capture the intricate graph structure of real-world banking networks. We thank the associate editor for pointing this out.}

Each miner is characterized by their wealth  and chooses their hash rate to maximize expected utility at a fixed time horizon.
Their wealth changes because of the mining rewards and expenses.
The instantaneous probability of receiving the reward is given by the probability of producing the next block:
\begin{equation*}\footnotesize
    \begin{aligned}
        \frac{\text{pl.\ $i$'s hash rate}}{\text{total hash rate}} &=
        \frac{\text{pl.\ $i$'s hash rate}}{\text{\#players $\times$ mean hash rate}} \approx
        \frac{\text{pl.\ $i$'s hash rate}}{\text{pl.\ $i$'s hash rate} + \text{(\#players$-1$) $\times$ mean hash rate}},
    \end{aligned}
\end{equation*}%
where the last expression is a good approximation when the number of players is large.
Let $M = \text{(\#players$-1$)}$, which is assumed to be large.
To formulate the continuum model, we replace the second term in the denominator on the right hand side by $(M \times \text{continuum mean hash rate})$.
With this approximation, the continuum game we consider has significant computational advantages over finite player games, which are often infeasible to solve numerically (because of dimensionality), especially in a dynamic setting.%
  \footnote{Mean field game structures and their computational advantages are briefly introduced and described in Appendix~\ref{app:mfg}.}
In other words, we mathematically model a continuum of players, but their aggregate effect on each other is still of size $M$, and is thus \emph{interpreted} as a game of $M+1$ players.\footnote{This is reasonable as long as the proportion of active players multiplied by $M$ remains large. We restrict our analysis to finite horizons for which this is the case.}

The purpose of using a \emph{continuum} mean field games structure instead of a finite $N$-player\footnote{We use $N$ here to talk about a general game and $N\to\infty$ so as not to confuse with $M$ which stays finite in our mining game.} model is that it significantly reduces the dimensionality.
The observation that this structure can admit simpler solutions, along with analysis of the convergence for $N\to\infty$, was pioneered in \cite{huang2006large,lasry2007mean}.
Whereas a finite $N$-player model requires solving $N$ coupled equations, the mean field games system consists of only two equations.
The $N$-dimensional system of nonlinear partial differential-difference equations with $N$ state variables in the first case is numerically intractable for $N\geq3$ players, while, as we demonstrate in this paper, the two continuum equations (in time and one state variable) are quite tractable for explicit and numerical resolution.
The intuition is as follows.

In a game of many, but finitely many, players, anyone considering the average of the others will observe a quantity very close to the true mean, by the law of large numbers.
As the number of players increases to infinity, the observed average converges to the true mean.
Hence, in each player's optimization problem, instead of solving for all possible combinations of randomness influencing others, they need only account for the true mean where individual fluctuations average out.

In a Markovian setting, the best response to the population average behavior can be written as a function of the state.
As two different players with the same state respond identically to the same population average, we simultaneously solve for the response of all players.
Knowing the best response, we enter this into the Fokker--Planck (Kolmogorov forward) equation, whose solution is the evolution of the density of the population over the state space.
With the density and best response, the population average can be computed.
If this coincides with the population average for which we found the best response, we have found a (mean field) equilibrium.
Such a mean field games equilibrium typically constitutes an $\epsilon$-Nash equilibrium to the finite player game (with $\epsilon$ converging to 0 as $N \to \infty$). See \ e.g.\ \cite{huang2007large,nourian2013}.

To summarize:
By considering a continuum of players, they do not need to consider random fluctuations affecting other individuals, which reduces the number of dynamic programming equations to one.
Instead, another equation accounts for the population dynamics and behavior.
These equations are coupled by a fixed point equation, and at the fixed point no individual has anything to gain from deviating---an equilibrium.


We next describe the general structure of each individual miner's problem, followed by its mathematical formulation.
With the structure of each miner's action, we then characterize their interaction and the equilibrium condition.

\subsection{General structure of the mining problem}
\label{general_structure}
We consider a continuum of miners who competitively engage in Bitcoin mining over some finite time period $[t_0,T]$. 
The representative miner chooses a  hash rate $\alpha_{t}\geq0$,
incurring a linear cost $\costr{\alpha_t}$ per unit of time, where $c>0$, and $t\in[t_0,T]$.
This cost is interpreted as the cost of electricity, and is thus proportional to their hash rates\footnote{\cite{garratt2020fixed} find that fixed costs associated with acquisition and setup of mining equipment improves the resilience of the system. We leave this interesting extension of the model for future research.}, and, for simplicity, each miner incurs the same marginal cost $c$.\footnote{We consider some models with cost heterogeneity in Section~\ref{sec:cost_adv} and Appendix~\ref{sec:cost_heterogeneity}.}

There are two important features of the Bitcoin proof-of-work protocol: First, the system always generates a reward on an almost fixed frequency that does not depend on the total hash rate.
In fact, the system will adjust the difficulty to make a reward available every 10 minutes on average.\footnote{In reality, in the case of Bitcoin, this number is adjusted every 2016 blocks (about every two weeks). We make the simplifying assumption that this happens continuously in our model. Similarly, we shall assume that also the miners' hash rates may change continuously.}
So it is reasonable to model the total number of rewards in the system as a whole as a Poisson process with a constant intensity $\K > 0$, where $\K^{-1}$ is approximately $10$ minutes. 
Second, a miner's probability of receiving the next mining reward is proportional to the ratio of its hash rate to that of the population. 
Since the math puzzle needs to be solved by brute force, the more hash rate a miner contributes, the more likely it will obtain the reward. 

The number of rewards each miner can receive is modelled by a counting process $N=(N_{t})_{t \geq t_0}$ with jump intensity $\rewardr=(\rewardr_t)_{t \geq t_0} > 0$.\footnote{Formally,
$P[ N_{t + \Delta t} - N_t = 1 ] = \rewardr_t \Delta t + o(\Delta t)$ and $P[ N_{t + \Delta t} - N_t \geq 2 ] = o(\Delta t)$, see \cite{bremaud}.}
Let $M + 1$ be the total number of miners and $\meanhashr_{t}\geq0$ denote the mean hash rate across all miners.
Here, our model for the reward intensity at time $t$ as a function of an individual's hash rate $\alpha_t$ and the mean hash rate is
\begin{equation}\label{eqn:reward_rate}
    \rewardr_t := \left\{\begin{aligned}
   & \K\,\frac{\alpha_{t}}{(\alpha_{t}+M\meanhashr_{t})}, & \alpha_t>0,\\
& 0 & \alpha_t=0,
\end{aligned}\right.   
\end{equation}
and we use $M\meanhashr_{t}$ to approximate the total hash rate of other miners.

Each miner is is modeled as having negligible impact on the population's mean production. However, the factor $M$ in front of the mean hash rate  in \eqref{eqn:reward_rate} implies that the mean field interaction is strong, whereas often in the literature it is assumed to be small for computational and technical reasons.
We argue that for cryptocurrency problems, interaction with the total hash rate is essential in a realistic model.
Indeed, this does introduce numerical difficulties, for which we provide an effective algorithm in Section~\ref{sec:illiq_numerical}.

The miners have initial wealth $x$, distributed at time $t_0$ according to an initial density function $m_0$. 
Then, an individual miner's wealth process $X=(X_{t})_{t \geq t_0}$ follows 
\begin{equation}
\label{wealth process for individual miner}
dX_{t} = - \costr{\alpha_{t}}\,dt + r\,dN_{t},
\end{equation}
where $r$ is the value of the mining reward.
Successfully mining a block, i.e., appending the next block to the ledger, grants the miner a reward as compensation.
This reward has two parts.
The so-called block reward is set by the system as a fixed number per block.
In addition to the block reward, the miner also receives any transaction fees from transactions included in the appended block.
The value of the total reward is interpreted as the product of the cryptoasset price and its quantity.\footnote{The miners are also rewarded the transaction fees in successfully mined blocks.  We use the term reward and mining reward to refer to the total amount received, i.e., the block reward plus the transaction fees.}

Since our focus is on the strategic decision of miners and the centralization in the competition, we treat the reward as a constant.%
\footnote{Although in reality the mining reward is not constant due to regular decreases of the block reward, fluctuations in transaction fees, and fluctuations in the value of the cryptocurrency relative to the unit denominating costs, we believe our base model is a tractable starting point from which a lot of features of the cryptocurrency market can be seen.
As a first step to seeing the effect of stochastic reward, in Appendix~\ref{sec:stoch_price}, we introduce a model with stochastic reward and demonstrate that, in the setting considered there, miners react only to the current level of the price and do not factor in reward volatility.
  Moreover, we have run the numerical experiments also with deterministically halving rewards, and qualitative results such as the cutoff point in Proposition~\ref{cutoff of the wealth level} are still observable.
  }
Nevertheless, the number of bitcoins as a reward is set to decrease geometrically with 50\% reduction every 4 years approximately,
and the current block reward is 6.25 bitcoins plus transaction fees of a few percent of the block reward.
Although not presented, we have numerically considered this decreasing reward (and, for the sake of completeness, increasing) without seeing a qualitative change in the behavior.

\subsection{The miners' optimization problem\label{oneminerprob}}
With the dynamics of Section~\ref{general_structure}, we are ready to formulate the miners' risk-reward problems.
The setup of the continuum model is to consider the optimization problem of an individual miner, in response to any given action of the rest.
Let $\alpha = (\alpha_{t})_{t \geq t_0}$ be a Markovian control.
The process $\alpha$ can then be associated with a function $(t, X_t) \mapsto \alpha(t, X_t; \meanhashr)$ of the current state.
With such controls, the wealth process $X$ is a Markov process.
The objective of the representative miner is to maximize the expected utility at fixed terminal time $T$.%
  \footnote{Many economic models choose an infinite horizon utility of consumption criterion to remove time as a dimension in the problem.
  That saving does not occur with mean field games, but the state dimension has already been decreased from $M+1$ to 2: one plus the time dimension.
  In addition, one could argue that cryptocurrency miners are not operating on an infinite horizon.}
Let $U$ denote a strictly increasing and concave utility function.%
  \footnote{The miners are endowed with the same utility function and costs in our analysis, although we do consider some cases of different risk aversions and costs in Section~\ref{sec:cost_adv}.
    Miners with heterogeneous risk aversion parameters could be considered similarly to the heterogeneous costs in Appendix~\ref{sec:cost_heterogeneity}, at some added complexity.}
The miner's value function is
\begin{equation}
\label{base value function}
v(t_0,x; \meanhashr) = \sup_{\alpha\geq0}\E[U(X_{T})\mid X_{t_0} = x].
\end{equation}
We give a basic property of the value function.
\begin{lemma}
	\label{property of the utility function base case}
    For any time $t\in [t_{0},T]$ and fixed $\meanhashr > 0$, the value function $v(t,x; \meanhashr)$ is finite and strictly increasing in the wealth $x$.
\end{lemma}
The proof is given in Appendix~\ref{proof2.1}. The result is intuitive, because more wealth gives more flexibility to miners to choose their hash rates, and it will be useful below.

For a fixed mean hash rate $\meanhashr > 0$, we first write down the HJB
\begin{equation}
\label{base HJB}
\partial_t v + \sup_{\alpha\geq0}\left(-\costr{\alpha} \partial_x v + \frac{\K\alpha}{(\alpha+M\meanhashr_{t})}\Delta v\right) = 0,
\end{equation}
with terminal condition $v(T,x) = U(x)$, and where $$\Delta v= v(t,x+r;\meanhashr) - v(t,x;\meanhashr).$$ 
We assume that $v$ is a classical solution of \eqref{base HJB}.\footnote{In the case of exponential utility, this is indeed verified in Section~\ref{sec:liquid}, so this is a reasonable assumption.}
By Lemma \ref{property of the utility function base case}, $\Delta v > 0$ and $\partial_x v > 0$, and so the optimal hash rate is given by
\begin{equation}
    \opthashr(t, x; \meanhashr) = \left\{
\begin{aligned}
    &-M\meanhashr_t + \sqrt{\frac{\K M \meanhashr_{t}\Delta v(t,x;\meanhashr)}{c\partial_x v(t,x;\meanhashr)}}, ~~~&&\text{if}~~~ \meanhashr_t < \frac{\K\Delta v(t,x;\meanhashr)}{Mc\partial_x v(t,x;\meanhashr)}, \\
\label{base optimal hash rate}
&0 &&\text{otherwise}.
\end{aligned}
\right.
\end{equation}
We observe in this expression that an individual miner at least temporarily stops hashing at time $t$ if the aggregate competition $M\meanhashr$ is above the threshold $\frac{\K}{c}\frac{\Delta v(t,x; \meanhashr)}{\partial_x v(t,x; \meanhashr)}$, but may resume as the game evolves. In this sense, there is endogenous entry and exit.

A key question here is whether at a given time $t$ there exists a cutoff point $x_{\mathrm{off}}(t)$ that separates the population into those who are hashing and those who are not, at that time.
We observe that such a threshold $x_{\mathrm{off}}(t)$ is determined by when $\frac{\Delta v(t,x; \meanhashr)}{\partial_x v(t,x; \meanhashr)}$ is small enough.
We will see that with CARA utility, the answer is no, whereas with CRRA utility the answer is yes, as shown in Proposition~\ref{cutoff of the wealth level}.

Finally, we simplify the HJB equation \eqref{base HJB} by plugging in \eqref{base optimal hash rate} to give:
\begin{equation}
\left\{
\begin{aligned}
    &\partial_t v + \left(\sqrt{Mc\meanhashr_t \partial_x v}-\sqrt{\K\Delta v}\right)^{2} = 0, ~~~&&\text{if}~~~ \meanhashr_t < \frac{\K\Delta v}{Mc\partial_x v},\\
\label{base HJB after optimization}
&\partial_t v = 0, &&\text{otherwise}.
\end{aligned}
\right.
\end{equation}

\subsection{Equilibrium characterization}\label{sec:equilibrium}
We now look for a Markovian equilibrium of the mining game.
It is useful to think of the continuum of miners as being labeled by their wealth $x$.
Let $\opthashr(t, x; \meanhashr)$ be the optimal hash rate of miner $x$,
and denote by $m(t,x; \meanhashr)$ the resulting density of the miners' wealth as a function of time and wealth.
We say $\eqmeanhashr$ forms an \emph{equilibrium mean hash rate} of the mining game if
\begin{equation}\label{eqn:equilibrium_mean_hash_rate}
\eqmeanhashr_t = \int_{\R}\opthashr(t,x; \eqmeanhashr)m(t,x; \eqmeanhashr)dx, ~~~\forall t \in [t_{0},T].
\end{equation}
Henceforth, let $\eqmeanhashr$ denote an equilibrium mean hash rate, and denote
\[
    v(t, x) = v(t, x; \eqmeanhashr), \quad \opthashr(t, x) = \opthashr(t, x; \eqmeanhashr), \quad m(t, x) = m(t, x; \eqmeanhashr).
\]

We assume that $\eqmeanhashr_t \neq 0$ for all $t$ for the following reason.
If $\eqmeanhashr_t = 0$, then each miner has an admissible control that dominates the choice of not mining.\footnote{\label{footnote_zero_mining}This holds for any cost $c$, as for any small $\varepsilon >0$, a miner can hash at rate $\alpha = \varepsilon$ and, on average, receive $r\K - c\varepsilon$ as net rewards.  As $\varepsilon$ is arbitrarily small, this is clearly positive with arbitrarily small risk.}
Hence, unless the mass of miners with non-zero admissible controls is zero, some mining will always occur.

We assume the initial density $m_{0}(x)$ is continuously differentiable and satisfies
\begin{equation}\label{eqn:initial_density}
\int m_{0}(x)\,dx = 1.
\end{equation}

Substituting \eqref{base optimal hash rate} into the equilibrium condition \eqref{eqn:equilibrium_mean_hash_rate} gives
\begin{equation}
\eqmeanhashr_t = \int_{E_t}\alpha^{*}(t,x)m(t,x)dx = -M\eta(t)\eqmeanhashr_t + \sqrt{\frac{\K M\eqmeanhashr_t}{c}}\int_{E_t}\sqrt{\frac{\Delta v(t,x)}{\partial_x v(t,x)}}m(t,x)\,dx,
\end{equation}
where 
\begin{equation}
\label{level set for active miners}
E_t = \{x: \alpha^{*}(t,x) > 0\}
\end{equation}
denotes the wealth level on which the miners are active and
$\eta(t) = \int_{E_t}m(t,x)\,dx$ 
denotes the fraction of active miners. Thus, 
\begin{equation}
\label{mean hash rate base case}
\eqmeanhashr_t = \frac{M}{(1+M\eta (t))^{2}}\left(\int_{E_t}\sqrt{\frac{\K\Delta v(t,x)}{c\partial_x v(t,x)}}m(t,x)\,dx\right)^{2},
\end{equation}
while the Fokker-Planck equation for $m$ is given by
\begin{equation}
\label{base fokker planck equation}
\partial_t m - \partial_x (c\alpha^{*}(t,x)m) -\K\left(\frac{\alpha^{*}(t,x-r)}{\alpha^{*}(t,x-r)+M\eqmeanhashr_t}m(t,x-r)-\frac{\alpha^{*}(t,x)}{\alpha^{*}(t,x)+M\eqmeanhashr_t}m(t,x)\right) = 0,
\end{equation}
with initial distribution $m(t_{0},x) =m_{0}(x)$.

Proving existence and uniqueness of an equilibrium in the general case is beyond the scope of this paper.
However, in Propositions~\ref{base case exponential utility} and~\ref{prop:cost_adv_exp_liq}, we find the unique equilibrium (in a given class) explicitly.
In the next section we introduce a numerical method whose results in Sections~\ref{sec:illiquid} and~\ref{sec:cost_adv_crra} give strong indication of existence of a unique equilibrium in those cases.
Moreover, in the case where we have an explicit solution, the numerical solution is observed to reproduce it (Section~\ref{sec:exp_num_conv}).

\subsection{Numerical method}\label{sec:illiq_numerical}

Solving the equations \eqref{base HJB after optimization} for $v$ and \eqref{base fokker planck equation} for $m$, \emph{individually}, are classical PDE problems.
The difficulty here lies in finding a solution that satisfies the equilibrium condition \eqref{eqn:equilibrium_mean_hash_rate}.
The equilibrium can be represented as a fixed point equation involving the solutions $v$ and $m$, formally written as
\[
  \meanhashr \mapsto \Psi(\meanhashr), \quad \text{where} \quad \Psi(\meanhashr) = (\text{integrate $\opthashr$, $m$}) \circ (\text{solve for $m$}) \circ (\text{solve for $v$, $\opthashr$})(\meanhashr),
\]
and where $\circ$ as usual denotes composition.
Our method for finding an equilibrium is an iteration of the fixed point mapping $\Psi$.

Each component of $\Psi$ can be readily computed with standard finite difference methods.
Although a direct fixed point iteration for $\Psi$ often does work, in this case, the large factor $M$ causes some overshoot when iterating.
To resolve this, we use an (under-) relaxation of the fixed point iteration.
Such relaxations are well known to help convergence and do not alter the fixed point (equilibrium) itself.
\cite{li2021simple} analyze convergence of this type of relaxed iterations in the particular case of mean field game problems.
We refer to them for complete technical details on this type of scheme and the mathematical benefits of relaxation.

What follows is a more detailed account of the iteration, the PDE solutions, and the relaxation technique.

\begin{enumerate}
    \item Initialize with a mean hash rate $t \mapsto \meanhashr_{t}$, for instance as constant.

    \item \label{num_step_v}
        Solve for the value function $v$ and the hash rate $\opthashr$:

        At time $T$, the value function is $v$ known, so \eqref{base optimal hash rate} yields $\opthashr(T, x; \meanhashr)$.
        This value is used as an approximation of $\opthashr(T-dt, x; \meanhashr)$, which allows solving for $v$ at $T-dt$, using the HJB equation:
        \begin{equation}
            \partial_t v + \left(-\costr{\opthashr(T, x; \meanhashr) } \partial_x v + \frac{\K\opthashr(T, x; \meanhashr) }{(\opthashr(T, x; \meanhashr)+M\meanhashr_{T})}\Delta v\right) = 0.
        \end{equation}
        The $\Delta v$ term is calculated explicitly using $v(T,x+r;\meanhashr) - v(T,x;\meanhashr)$, while the other part is discretized by an implicit finite difference scheme.
        This separation is made to balance stability with computational efficiency.
        With the value function $v$ at $T-dt$, we can get $\opthashr(T-dt, x; \meanhashr)$.
        Repeat such time steps backwards until $t = 0$.
        This yields both functions $v$ and $\opthashr$.

    \item \label{num_step_m}
        The next step is to solve the Fokker--Planck equation and get the mean field control.

        The $\opthashr(t,x; \meanhashr)$ is obtained from the previous step allows solving \eqref{base fokker planck equation} and \eqref{eqn:fokkerplanck_illiq_smallx} for $m(t,x)$.
        In doing so, the following terms are discretized by an implicit finite difference scheme
        \begin{align}
            \partial_t m  - \partial_x (c\alpha^{*}(t,x)m) + \frac{\K\alpha^{*}(t,x)}{(\alpha^{*}(t,x) + M \meanhashr_{t})}m,
        \end{align}
        while $m$ in
        \begin{align}
            -\frac{\K\alpha^{*}(t,x-r)}{\alpha^{*}(t,x-r)+M\meanhashr_t}m(t,x-r)
        \end{align}
        is evaluated in the previous time step, again separated to balance stability and efficiency.

    \item
      Due to the large factor $M$ scaling up the errors away from the equilibrium, so a direct iteration will tend to overshoot.
      We therefore use an under-relaxation technique to reduce this effect.
      For some parameter $w\in[0,1)$ and for each time $t$, the relaxed iteration is given by
      \begin{equation}
        \label{eqn:relaxed_mean_hash_rate}
          \meanhashr_{t}^\text{new} = w\meanhashr_{t} + (1-w)\int_{\R}\opthashr(t,x; \meanhashr)m(t,x; \meanhashr)dx = w \meanhashr_{t} + (1-w)\Psi(\meanhashr)_t.
      \end{equation}
      This type of (under-) relaxation technique is common in iteration schemes and stabilizes the iteration at the expense of slower iteration, cf.\ successive over-relaxation, which instead speeds up stable iterations with $w\geq 1$.
      The choice of $w$ has no impact on the equilibrium fixed point.

    \item
      Finally, repeat from the first step with $\meanhashr = \meanhashr^\text{new}$, unless the difference is sufficiently small to terminate.%
      \footnote{%
        As usual with fixed point iterations, convergence requires the initial guess to lie in the basin of attraction of the fixed point mapping.
        Outside this set, the iterations soon result in $\meanhashr = 0$, which can be discarded as a non-equilibrium, as mentioned in Section~\ref{sec:equilibrium}.
        The choice $w = 1 - \frac{1}{M}$ makes the set large enough to find after a few guesses, and this is the value used for all figures.
        We emphasize that in all our experiments the limiting equilibrium has always been unique, regardless of initial guess, for any particular set of parameters.
      }
\end{enumerate}

We note that \cite{li2021simple} study relaxation of the problem by relaxing the updating of both $v$ and $m$.
Adding relaxation to the updating of $v$ and $m$ (steps \ref{num_step_v} and \ref{num_step_m}) could have further benefits in improving convergence properties, but because the scheme here successfully solved all problems we considered, such further relaxation was not explored.

\section{Explicit (CARA utility) and Numerical (CRRA utility) Solutions \& Analysis}\label{sec:results}
When the miners' utility functiona are of CARA-type,
we give explicit formulas for the equilibrium Section (\ref{sec:liquid}).%
These provide evidence that the model is well-posed, but, as usual, wealth effects are missing with this choice of utility.
We demonstrate in Section~\ref{sec:exp_num_conv} that the numerical method reproduces the explicit solution.
We also address the problem with miners having CRRA utilities numerically, in which case wealth effects become apparent.

\subsection{Exponential utility (CARA)}\label{sec:liquid}
With exponential utility, 
\begin{equation}
  \label{eqn:exp_util}
  U(x) = -\frac{1}{\gamma}e^{-\gamma x}, \quad \gamma>0, \quad x\in\R,
\end{equation}
where $\gamma$ is the risk aversion parameter,\footnote{With a slight abuse of notation, we use $\gamma = 0$ to label the risk-neutral case of $U(x) = x$.
  The results of Section~\ref{sec:liquid} hold for risk-neutral miners with formal replacement of $\gamma$ by 0.}, 
a miner's current wealth $X_{t}$ can take any value in $\R$.  
It is convenient to define\footnote{In the risk-neutral case, $w(r; 0) = r$.}
\begin{equation} 
w(r;\gamma) =  \frac{1-e^{-\gamma r}}{\gamma}.\label{wdef}
\end{equation}

\subsubsection{Explicit solution}
The following proposition gives formulas for the explicit solution of the mining game.

\begin{proposition}
	\label{base case exponential utility}
	With exponential utility, \eqref{eqn:exp_util}, in the equilibrium, all miners are always active, with constant hash rate
	\begin{equation}
	\label{exponential optimal rate}
	\alpha^{*}(t, x) \equiv \eqmeanhashr_t \equiv \frac{
	\K M}{c(1+M)^{2}} w(r;\gamma),
	\end{equation}
    and their individual reward rate is
    \[
        \rewardr_t \equiv \frac{\K}{(1+M)},
    \]
	for any $t\in [t_{0},T]$ and $x\in \R$. The value function is given by
	\begin{equation}
	\label{exponential value function}
	v(t,x) = U(x)e^{-\gamma w(r;\gamma)\frac{\K(T-t)}{(1+M)^{2}}}.
	\end{equation}
\end{proposition}

\begin{remark}
    \label{rem:mfg_finite_coincide}
    Due to the wealth independence of optimal hash rates, and thus the wealth distribution, this is the same solution as a model without the mean field game approximation.
    Indeed, with $M+1$ players all using the same (wealth-independent) strategy $\opthashr$, the total hash rate in the mean field games model $\opthashr + M \eqmeanhashr = (1 + M) \opthashr$ is equal to total hash rate in the finite player model: $\sum_{i=0}^M \opthashr = (1 + M) \opthashr$.
    This exact correspondence is specific to this particular setup.
\end{remark}

\begin{remark}
    By the above solution, we observe that the total mining $\sum_{i=1}^M \alpha^{*}(t, x)$ is bounded by, and for large $M$ approximately equal to $\frac{\K}{c}w(r;\gamma)$.
    This is in turn bounded by $\K/c\gamma$ for CARA utility,
    and, as a consequence, the total mining is thus also bounded by $\K/c\gamma$, regardless of $r$.
\end{remark}

\begin{proof}
	We guess the form $v(t,x) = U(x)h(t)$ and then 
    the HJB equation \eqref{base HJB after optimization} becomes
	\begin{equation}
	\left\{
	\begin{aligned}
	&\partial_{t}h - \gamma \left(\sqrt{Mc\eqmeanhashr_t}-\sqrt{\K w(r;\gamma)}\right)^{2}h = 0, ~~~&&\text{if}~~~ \eqmeanhashr_t< \frac{\K}{Mc}w(r;\gamma),\\
	&\partial_{t}h = 0, &&\text{otherwise},
	\end{aligned}
	\right.
	\end{equation}
	with terminal condition $h(T) = 1$.
	It is an equation in $t$ only, which validates our ansatz. Moreover, this ansatz implies that $\Delta v$/$\partial_x v$ does not depend on $x$. Hence, from \eqref{mean hash rate base case},
	\begin{equation}
	\eqmeanhashr_t = \frac{\eta^{2}M}{(1+\eta M)^{2}}\frac{\K\Delta v}{c\partial_x v}\leq \frac{M}{(1+M)^{2}}\frac{\K\Delta v}{c\partial_x v} < \frac{\K\Delta v}{Mc\partial_x v},
	\end{equation}
	which means all miners are active and $\eta \equiv 1$. 
	
	The equilibrium mean hash rate is obtained from plugging in the ansatz into \eqref{mean hash rate base case}.
    Thereafter, \eqref{base optimal hash rate} yields \eqref{exponential optimal rate}.
    The value function then satisfies
	\begin{equation}
	\partial_{t}h - \frac{\gamma\K}{(1+M)^{2}}w(r;\gamma)\,h = 0,
	\end{equation}
	with terminal condition $h(T) = 1$. Thus we have \eqref{exponential value function}.
\end{proof}

\subsubsection{Risk-Reward Analysis}\label{sec:risk_reward_analysis}
As $\opthashr$ is constant, we drop the dependence on $t$ and $x$.
In the equilibrium, the wealth of the representative miner can be written as 
\begin{equation}
    X_{t_0 + t} = X_{t_0} - \costr{\opthashr} t + r (N_{t_0 + t}^{*} - N_{t_0}^*) = X_{t_0} - \frac{\K M}{(1+M)^{2}} w(r;\gamma)t + r (N_{t_0 + t}^{*} - N_{t_0}^*),
\end{equation}
where $N^{*}$ has the constant jump rate $\lambda_t = \frac{\K}{(1+M)}$ and thus is a Poisson process.
Then we can compute the expectation and variance of $X_{t_0 + t}$:
\begin{align}
    \E[X_{t_0 + t} - X_{t_0}] = \Big(r-\frac{M w(r;\gamma)}{1+M} \Big)\frac{\K t}{(1+M)},
    \quad
    \text{Var}(X_{t_0 + t} - X_{t_0}) = \frac{\K r^{2}t}{(1+M)}.
\end{align}

If $\gamma r$ is small, the expected wealth change is approximately
$\frac{\K r t}{(1+M)^{2}}$,
which decays with competition as $M^{-2}$.
However, the variance is only discounted by a linear factor $1+M$, and in particular the standard deviation decays as $1/\sqrt{M}$.
Hence, as more miners join the game, the expected gain shrinks very fast, but the potential risk decreases slowly: a more crowded competitive mining environment is less attractive in terms of the risk-reward tradeoff than a sparser one.

\subsubsection{Numerical convergence}
\label{sec:exp_num_conv}
We use the explicit solution \eqref{exponential optimal rate} to test the numerical scheme from Section~\ref{sec:illiq_numerical}.
Table~\ref{tab:exp_num_errors} lists errors for one run.
The errors are observed to converge toward zero,
and as the iteration closes in on the correct solution, the convergence is of order 1 and rate 0.5.
Similar numbers are obtained for other starting points and parameter values, although the initial iterations can be slower than rate 0.5 far away from the solution, due to the strong under-relaxation.
The relative error of approximately 0.1\% in the last column is due to discretization and decreases with the grid size.

This observed convergence indicates that the numerical method is sound, and is therefore valuable indication of the correctness in subsequent sections with CRRA utility, where no explicit solution exists.

\begin{table}
  \centering
 \small
\begin{tabular}{c|c|c}
  Iteration & $\frac{1}{T} \sum_{t=0}^T |\meanhashr_t - \meanhashr_t^{\text{new}}|$ & $\frac{1}{T} \sum_{t=0}^T (\Psi(\meanhashr)_t - \meanhashr_t^{\text{explicit}}) / \meanhashr_t^{\text{explicit}}$ \\[1em]
  \hline
  1 & 266244 & 69.673 \\
  2 & 139090 & 36.398 \\
  3 & 71891 & 18.813 \\
  4 & 36485 & 9.5473 \\
  5 & 18382 & 4.8095 \\
  6 & 9228 & 2.4139 \\
  7 & 4624 & 1.2089 \\
  8 & 2314 & 0.60440 \\
  9 & 1157 & 0.30166 \\
  10 & 578.6 & 0.15020 \\
  \hline
  11 & 289.28 & 0.07446 \\
  12 & 144.60 & 0.03659 \\
  13 & 72.281 & 0.01766 \\
  14 & 36.130 & 0.00820 \\
  15 & 18.060 & 0.00347 \\
  16 & 9.0285 & 0.00111 \\
  17 & 4.5135 & 0.0000068 \\
  18 & 2.2565 & 0.00065 \\
  19 & 1.1282 & 0.00095 \\
  20 & 0.5641 & 0.00110 \\
  \hline
  30 & 0.00055 & 0.00124 \\
  \hline
  40 & 5.5184e-07 & 0.00124 \\
  \hline
  50 & 4.2964e-09 & 0.00124 \\
\end{tabular}
\caption{\label{tab:exp_num_errors}\small{\em Example of observed convergence of the numerical method to the explicit solution for CARA utility.
  Following the notation in Section~\ref{sec:illiq_numerical}, $\meanhashr$ denotes the output of the previous iteration, $\Psi(\meanhashr)$ denotes the fixed point mapping output of the current iteration, $\meanhashr^\mathrm{new}$ is the relaxed output \eqref{eqn:relaxed_mean_hash_rate}, and $\meanhashr^\mathrm{explicit}$ denotes the explicit solution from Proposition~\ref{base case exponential utility}.
  The error remaining in the last column is due to the discretization error of the PDEs and decreases with the grid size.}}
\end{table}

\subsection{Power Utility (CRRA)\label{sec:illiquid}}%
In this section, we consider the mining problem where miners have CRRA utility functions $U$, defined on $\Rnneg$, namely
\begin{equation}
\label{mining power utility}
U(x) = \frac{1}{1-\gamma}x^{1-\gamma} ~~~\text{for}~\gamma \in (0,1),\quad x>0,
\end{equation}
and admissible strategies $\alpha$ are such that the wealth process $X$ remains positive.
The HJB equation \eqref{base HJB after optimization} holds on $x > 0$ with the boundary condition $V(t, 0) = U(0) = 0$.
We naturally assume that the initial density $m_0$ has strictly positive support, so all the miners start with positive wealth.
This means that the problem is fully characterized on $\Rnneg$.

In contrast to the case of exponential utility in Section~\ref{sec:liquid}, 
the value function cannot be found explicitly, even with power utility, so we must solve \eqref{base HJB after optimization} numerically.
The mean hash rate is still given by \eqref{mean hash rate base case}, while the Fokker--Planck equation has two parts to account for the fact that miners' wealths stay strictly positive.
The density $m$ solves \eqref{base fokker planck equation} (at least in a weak sense) for $x>r$.
On the other hand, if $0<x<r$, there is no density at $x-r$ jumping to $x$, because there are no miners with negative wealth.
So, for $0 < x < r$, the density $m$ solves
\begin{equation}\label{eqn:fokkerplanck_illiq_smallx}
\partial_t m  - \partial_x (c\alpha^{*}m) + \frac{\K\alpha^{*}}{(\alpha^{*}+M\eqmeanhashr)}m = 0,
\end{equation}
where the optimal individual hash rate is given by \eqref{base optimal hash rate}.
The initial condition on all $x>0$ is $m(t_{0},x) = m_{0}(x)$.

\begin{remark}
    The Fokker--Planck equation can be verified to preserve mass on $\Rnneg$, i.e., for all $t \in [t_0, T]$, $\int_{\Rnneg} m(t, x) dx = 1$.
    Indeed,
	\begin{align}
	\partial_{t}\int_{\Rnneg}m(t,x)dx  &= \int_{\Rnneg}\partial_t m(t,x)\,dx, \\ 
    &= c\alpha^{*}(t, \cdot) m(t, \cdot) |_{0}^{r} - \int_{0}^{+\infty}\frac{\K\alpha^{*}}{(\alpha^{*}+M\eqmeanhashr_{t})}m\,dx + c\alpha^{*}(t, \cdot)m(t, \cdot) |_{r}^{+\infty}\\
    &\quad  + \int_{r}^{+\infty}\frac{\K\alpha^{*}(t,x-r)}{(\alpha^{*}(t,x-r)+M\eqmeanhashr_{t})}m(t,x-r)\,dx, \\
    &= c\alpha^{*}(t, \cdot) m(t, \cdot) |_{0}^{+\infty} = 0,
	\end{align}
	where the second integral term cancels with the first by making the change of variable $x'=x-r$.
	The last equation holds because no one can obtain infinite wealth in finite time, and the condition $\alpha^{*}(t,0^+) = 0$ holds. 
\end{remark}

\subsubsection{Concentration of wealth and mining effort}
\label{sec:concentration_of_wealth}
In this section, we numerically solve for the equilibrium with power utility functions \eqref{mining power utility}.
This structure leads to strategic decisions of the miners that are very different from those found with exponential utility, as those with larger wealth tend to hash more.
This also puts competitive pressure on lower wealth miners, and as a consequence, the wealthier miners receive a disproportionately large share of profits, a form of \emph{preferential attachment}.

Preferential attachment effects appears in many situations: scientific citation networks \citep{Barabsi2002}, language use \citep{Perc2012}, distribution of cities by population and distributions of incomes by size \citep{Simon1955}.
A recent empirical study points out that it also appears in the Bitcoin network \citep{Kondor2014}: 
``we find that the wealth of already rich nodes increases faster than the wealth of nodes with low balance.''
This empirical observation is consistent with our model's numerical results.

Figure \ref{fig:The distribution of miners' wealth.} shows the distribution of the miners' wealth at $t= 30, 45, 60, 90$ compared with the initial distribution.
As time increases, the majority of the mass moves to the left, gradually forming a significant spike.
At the same time, there is small part of the mass moving to the right. This indicates that most miners lose their wealth, but those who have relatively more money originally accumulate wealth over time.

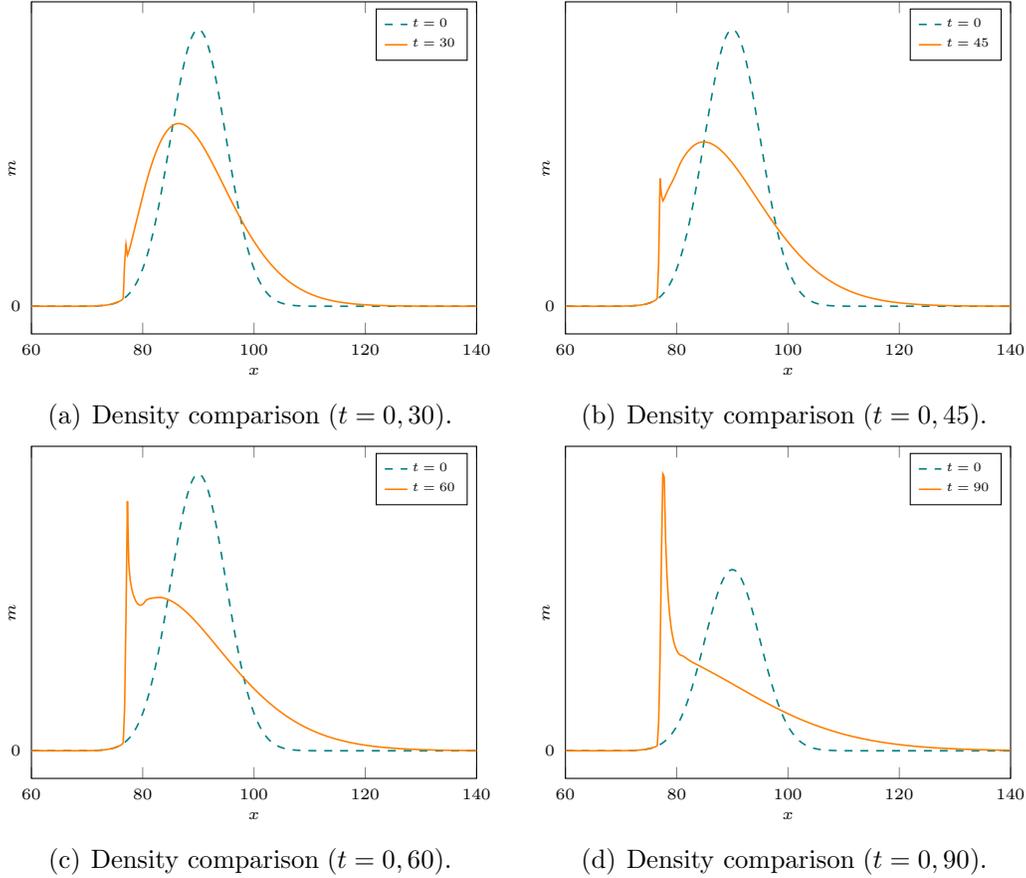
\begin{figure}[htbp]
	\centering
    \newcommand{\axissettings}{xlabel=$x$, ylabel = $m$, enlarge x limits=0, ytick={0}, ylabel style={yshift=-0.5cm}}
    \newcommand{\legendsettings}{legend cell align={left}, legend image post style={scale=0.5}, legend style={nodes={scale=1.2}}}
        \begin{tikzpicture}
            \begin{axis}[width=0.85\textwidth, \axissettings,\legendsettings]
                \addplot[teal, loosely dashed, line width=0.6pt, restrict x to domain=60:140] table [x=x, y=m0, col sep=comma] {Data/m.csv};
                \addlegendentry{$t=0$}
                \addplot[orange, line width=0.6pt, restrict x to domain=60:140] table [x=x, y=m30, col sep=comma] {Data/m.csv};
                \addlegendentry{$t=30$}
                \addplot[red, densely dotted, line width=0.6pt, restrict x to domain=60:140] table [x=x, y=m45, col sep=comma] {Data/m.csv};
                \addlegendentry{$t=45$}
                \addplot[blue, densely dashed, line width=0.6pt, restrict x to domain=60:140] table [x=x, y=m60, col sep=comma] {Data/m.csv};
                \addlegendentry{$t=60$}
                \addplot[darkgray, loosely dotted, line width=0.6pt, restrict x to domain=60:140] table [x=x, y=m90, col sep=comma] {Data/m.csv};
                \addlegendentry{$t=90$}
            \end{axis}
        \end{tikzpicture}
	\caption{\small{\em The distribution of miners' wealth at different times. Parameters: $\K^{-1} = 0.007$, $r=3$, $c=2\times 10^{-5}$, $T=90$, $\gamma=0.8$, $M=1000$.
    The initial distribution $m_{0}$ is normal with mean $90$ and standard deviation $5$.
	  Note that as time passes, there is mass concentrating on the left hand side of the peak at $x = x_{\mathrm{off}} \approx 77$, cf.\ Proposition~\ref{cutoff of the wealth level}.
	  In light of Proposition~\ref{cutoff of the wealth level}, this is expected.
	  Below the peak, it is optimal for miners to not participate, whereas above they do.
	  Hence, miners to the right of $x_{\mathrm{off}}$ move to the left, whereas those at or below are stationary, causing concentration of the density around $x_{\mathrm{off}}$.
	  We emphasize that this does not cause numerical problems, because derivatives are only evaluated on one side of this point.
 }}
	\label{fig:The distribution of miners' wealth.}
\end{figure}

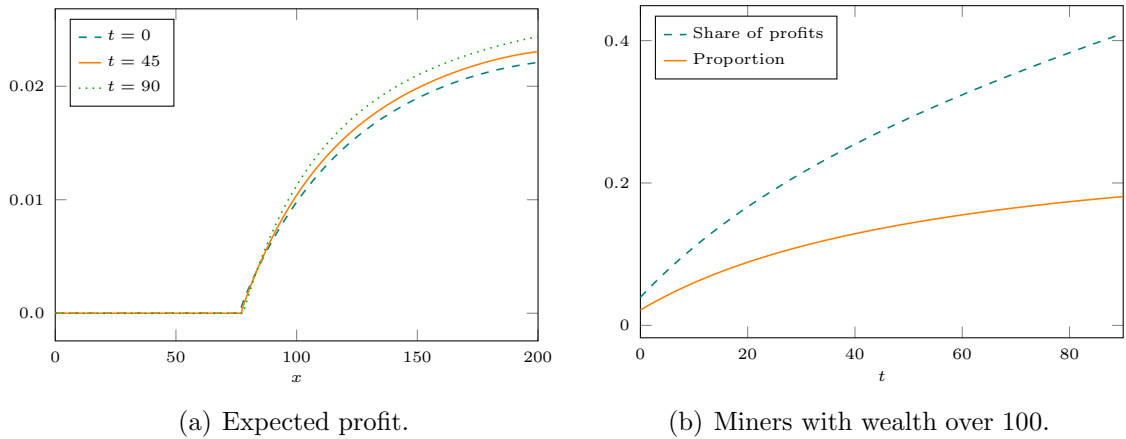
\begin{figure}
	\centering
    \newcommand{\axissettings}{width=8cm, height=6cm, enlarge x limits = 0}
    \newcommand{\legendsettings}{legend pos=north west, legend cell align={left}, legend image post style={scale=0.5}, legend style={nodes={scale=1.0}}}
	\subfigure[Expected profit.]{\label{fig:profitforx}
		\begin{tikzpicture}
            \begin{axis}[\axissettings,\legendsettings, 
            	y tick label style={
            	/pgf/number format/.cd,
            	precision=3,
            },
            yticklabel=\phantom{0}$\mathllap{\tick}$,
            legend pos=north west, xlabel=$x$, ylabel style={yshift=-0.5cm}, xlabel style={font=\tiny,at={(axis description cs:0.5,0.05)}},ylabel style={font=\tiny,at={(axis description cs:0.08,0.5)}}]
		\addplot[teal, dashed, line width=0.6pt] table [x=x, y=profit0, col sep=comma] {Data/profit.csv};
		\addlegendentry{$t=0$}
		\addplot[orange, line width=0.6pt] table [x=x, y=profit45, col sep=comma] {Data/profit.csv};
		\addlegendentry{$t=45$}
		\addplot[black!35!green, dotted, line width=0.6pt] table [x=x, y=profit90, col sep=comma] {Data/profit.csv};
		\addlegendentry{$t=90$}
		\end{axis}
		\end{tikzpicture}
    }
	\subfigure[Miners with wealth over $100$.]{\label{fig:prob_profit_100}
		\begin{tikzpicture}
            \begin{axis}[\axissettings,\legendsettings, xlabel=$t$, xlabel style={font=\tiny,at={(axis description cs:0.5,0.05)}},ylabel style={font=\tiny,at={(axis description cs:0.08,0.5)}}]
		\addplot[teal, dashed, line width=0.6pt] table [x=t, y=profit, col sep=comma] {Data/prop_vs_profit.csv};
		\addlegendentry{Share of profits}
		\addplot[orange, line width=0.6pt] table [x=t, y=prop, col sep=comma] {Data/prop_vs_profit.csv};
		\addlegendentry{Proportion}
		\end{axis}
		\end{tikzpicture}
    }
    \caption{\small{\em 
        The left plot shows the expected instantaneous profit miners at different wealth levels and times.
        The right one gives the proportion of miners with wealth over 100 and their share of the total instantaneous profits.
        Parameters are the same as Figure \ref{fig:The distribution of miners' wealth.}.
    }}
	\label{fig:The probability of getting the reward.}
\end{figure}

From \eqref{wealth process for individual miner}, the expected profit $\E[\Delta X_t]$ for an individual miner over a short time $\Delta t$ is given, approximately by $\E[\Delta X_t]\approx - c\alpha^*_t\,\Delta t + r\lambda_t\Delta t$, where $\lambda_t$ is given by \eqref{eqn:reward_rate}.
Figure \ref{fig:profitforx} shows the instantaneous profit rate, namely
\begin{equation}
    -c \alpha^*(t,x) + r\K\frac{\alpha^*(t, x)}{\alpha^*(t, x) + M \eqmeanhashr_t},	\label{instprofrate}
\end{equation}
as a function of miners' wealths at various times. 
We see that the more wealth a miner has, the higher is their expected profit rate.
Miners with lower wealth hash at lower rates or even zero rate below $x_{\mathrm{off}}$ introduced in Proposition~\ref{cutoff of the wealth level} below.
This pattern holds at all times.
In addition, it is interesting to see that the miners are blockaded ($0$ hash rate) around and below wealth level $77$ (see Figure \ref{fig:profitforx}).
Their risk aversion prevents miners with small wealth from participating in the mining game.
This is explained by the following lemma, whose proof is provided in Appendix~\ref{sec:proofs}.
\begin{proposition}
	\label{cutoff of the wealth level}
    For any time $t$ and equilibrium hash rate $\eqmeanhashr_t > 0$, there exists a wealth level $x_{\mathrm{off}}(t) > 0$ such that $\alpha^*(t, x) = 0$ for $x \leq x_{\mathrm{off}}(t)$, i.e., the optimal mining rate is zero at wealth levels below $x_{\mathrm{off}}$.
\end{proposition}

To better understand the preferential attachment effect, we compute the proportion of miners whose wealth is over $100$ and their share of the instantaneous profits,
both of which are shown in Figure \ref{fig:prob_profit_100}.
The proportion increases from around $2\%$ to $18\%$, while the share of profits rises from $4\%$ to $41\%$.
Hence, as time goes by, the wealthy receive an increasingly large share of the profits.

We have shown ``the rich get richer'' phenomenon in the mining game.
It is also interesting to see how the reward $r$ and competition parameter $M$ affect this phenomenon. 
To avoid repetition, we show only the plots at $t = 30$, but a similar pattern is present also at other times.

\begin{figure}[htbp]
	\centering
    \newcommand{\axissettings}{width=8cm, height=6cm, enlarge x limits = 0}
    \newcommand{\legendsettings}{legend cell align={left}, legend image post style={scale=0.5}, legend style={nodes={scale=1.0}}}
	\subfigure[Density comparison ($r=2, 3$).]{\label{fig:t30p23}
		\begin{tikzpicture}
            \begin{axis}[\axissettings,\legendsettings, xlabel=$x$, ylabel = $m$, ytick={0}, ylabel style={yshift=-0.5cm}, xlabel style={font=\tiny,at={(axis description cs:0.5,0.05)}},ylabel style={font=\tiny,at={(axis description cs:0.08,0.5)}}]
		\addplot[teal, dashed, line width=0.6pt, restrict x to domain=60:140] table [x=x, y=m_p2, col sep=comma] {Data/m30p2345.csv};
		\addlegendentry{$r=2$}
		\addplot[orange, line width=0.6pt, restrict x to domain=60:140] table [x=x, y=m_p3, col sep=comma] {Data/m30p2345.csv};
		\addlegendentry{$r=3$}
		\end{axis}
		\end{tikzpicture}
    }
	\subfigure[Density comparison ($r=3, 4$).]{\label{fig:t30p34}
		\begin{tikzpicture}
            \begin{axis}[\axissettings,\legendsettings, xlabel=$x$, ylabel = $m$, ytick={0}, ylabel style={yshift=-0.5cm}, xlabel style={font=\tiny,at={(axis description cs:0.5,0.05)}},ylabel style={font=\tiny,at={(axis description cs:0.08,0.5)}}]
		\addplot[teal, dashed, line width=0.6pt, restrict x to domain=60:140] table [x=x, y=m_p3, col sep=comma] {Data/m30p2345.csv};
		\addlegendentry{$r=3$}
		\addplot[orange, line width=0.6pt, restrict x to domain=60:140] table [x=x, y=m_p4, col sep=comma] {Data/m30p2345.csv};
		\addlegendentry{$r=4$}
		\end{axis}
		\end{tikzpicture}
    }
	\subfigure[Density comparison ($r=4, 5$).]{\label{fig:t30p45} 
		\begin{tikzpicture}
            \begin{axis}[\axissettings,\legendsettings, xlabel=$x$, ylabel = $m$, ytick={0}, ylabel style={yshift=-0.5cm}, xlabel style={font=\tiny,at={(axis description cs:0.5,0.05)}},ylabel style={font=\tiny,at={(axis description cs:0.08,0.5)}}]
		\addplot[teal, dashed, line width=0.6pt, restrict x to domain=60:140] table [x=x, y=m_p4, col sep=comma] {Data/m30p2345.csv};
		\addlegendentry{$r=4$}
		\addplot[orange, line width=0.6pt, restrict x to domain=60:140] table [x=x, y=m_p5, col sep=comma] {Data/m30p2345.csv};
		\addlegendentry{$r=5$}
		\end{axis}
		\end{tikzpicture}
    }
	\subfigure[Expected profit.]{\label{fig:prof30p2345} 
		\begin{tikzpicture}
            \begin{axis}[\axissettings,\legendsettings,
                y tick label style={
                    /pgf/number format/.cd,
                    precision=3,
                },
                yticklabel=\phantom{0}$\mathllap{\tick}$,
                legend pos=north west, xlabel=$x$, ylabel style={yshift=-0.5cm}, xlabel style={font=\tiny,at={(axis description cs:0.5,0.05)}},ylabel style={font=\tiny,at={(axis description cs:0.08,0.5)}}]
        \addplot[teal, dashed, line width=0.6pt] table [x=x, y=profit_p2, col sep=comma] {Data/profit30p2345.csv};
		\addlegendentry{$r=2$}
		\addplot[orange, line width=0.6pt] table [x=x, y=profit_p3, col sep=comma] {Data/profit30p2345.csv};
		\addlegendentry{$r=3$}
		\addplot[black!35!green, dotted, line width=0.6pt] table [x=x, y=profit_p4, col sep=comma] {Data/profit30p2345.csv};
		\addlegendentry{$r=4$}
		\addplot[black!35!red, loosely dashed, line width=0.6pt] table [x=x, y=profit_p5, col sep=comma] {Data/profit30p2345.csv};
		\addlegendentry{$r=5$}
		\end{axis}
		\end{tikzpicture}
    }
	\caption{\small{\em The price effects at $t = 30$. The first three plots show the distribution of miners' wealth.
        The last one shows the expected instantaneous profit of miners at different wealth levels, for four different price levels.
    Parameters are the same as Figure \ref{fig:The distribution of miners' wealth.} except that $r$ takes multiple values.}}
	\label{fig:price effect at 30.}
\end{figure}

A larger reward $r$ exacerbates the degree of preferential attachment.
The density plots in Figure \ref{fig:price effect at 30.} show that those with lower wealth tend to lose money faster when the price is higher.
At the same time, the density for $x\geq110$ is clearly higher for the larger reward in Figures \ref{fig:price effect at 30.} (a)(b)(c).
Figure \ref{fig:prof30p2345} shows the expected instantaneous profits \eqref{instprofrate}, which leads to the same conclusion.

\begin{figure}[htbp]
	\centering
    \newcommand{\axissettings}{width=8cm, height=6cm, enlarge x limits = 0}
    \newcommand{\legendsettings}{legend cell align={left}, legend image post style={scale=0.5}, legend style={nodes={scale=1.0}}}
	\subfigure[Density comparison ($M=1K, 2K$).]{\label{fig:t30M12}
		\begin{tikzpicture}
            \begin{axis}[\axissettings,\legendsettings, xlabel=$x$, ylabel = $m$, ytick={0}, ylabel style={yshift=-0.5cm}, xlabel style={font=\tiny,at={(axis description cs:0.5,0.05)}},ylabel style={font=\tiny,at={(axis description cs:0.08,0.5)}}]
		\addplot[teal, dashed, line width=0.6pt, restrict x to domain=60:140] table [x=x, y=m_M1, col sep=comma] {Data/m30M12310.csv};
		\addlegendentry{$M=1000$}
		\addplot[orange, line width=0.6pt, restrict x to domain=60:140] table [x=x, y=m_M2, col sep=comma] {Data/m30M12310.csv};
		\addlegendentry{$M=2000$}
		\end{axis}
		\end{tikzpicture}
    }
	\subfigure[Density comparison ($M=2K, 3K$).]{\label{fig:t30M23}
		\begin{tikzpicture}
            \begin{axis}[\axissettings,\legendsettings, xlabel=$x$, ylabel = $m$, ytick={0}, ylabel style={yshift=-0.5cm}, xlabel style={font=\tiny,at={(axis description cs:0.5,0.05)}},ylabel style={font=\tiny,at={(axis description cs:0.08,0.5)}}]
		\addplot[teal, dashed, line width=0.6pt, restrict x to domain=60:140] table [x=x, y=m_M2, col sep=comma] {Data/m30M12310.csv};
		\addlegendentry{$M=2000$}
		\addplot[orange, line width=0.6pt, restrict x to domain=60:140] table [x=x, y=m_M3, col sep=comma] {Data/m30M12310.csv};
		\addlegendentry{$M=3000$}
		\end{axis}
		\end{tikzpicture}
    }
	\subfigure[Density comparison ($M=3K, 10K$).]{\label{fig:t30M310}
		\begin{tikzpicture}
            \begin{axis}[\axissettings,\legendsettings, xlabel=$x$, ylabel = $m$, ytick={0}, ylabel style={yshift=-0.5cm}, xlabel style={font=\tiny,at={(axis description cs:0.5,0.05)}},ylabel style={font=\tiny,at={(axis description cs:0.08,0.5)}}]
		\addplot[teal, dashed, line width=0.6pt, restrict x to domain=60:140] table [x=x, y=m_M3, col sep=comma] {Data/m30M12310.csv};
		\addlegendentry{$M=3000$}
		\addplot[orange, line width=0.6pt, restrict x to domain=60:140] table [x=x, y=m_M10, col sep=comma] {Data/m30M12310.csv};
		\addlegendentry{$M=10000$}
		\end{axis}
		\end{tikzpicture}
    }
	\subfigure[Expected profit.]{\label{fig:profit30M12310}
		\begin{tikzpicture}
            \begin{axis}[\axissettings,\legendsettings,
                y tick label style={
                    /pgf/number format/.cd,
                    precision=3,
                },
                yticklabel=\phantom{0}$\mathllap{\tick}$,
                legend pos=north west, xlabel=$x$, ylabel style={yshift=-0.5cm}, xlabel style={font=\tiny,at={(axis description cs:0.5,0.05)}},ylabel style={font=\tiny,at={(axis description cs:0.08,0.5)}}]
        \addplot[teal, dashed, line width=0.6pt] table [x=x, y=profit_M1, col sep=comma] {Data/profit30M12310.csv};
		\addlegendentry{$M=1000$}
		\addplot[orange, line width=0.6pt] table [x=x, y=profit_M2, col sep=comma] {Data/profit30M12310.csv};
		\addlegendentry{$M=2000$}
		\addplot[black!35!green, dotted, line width=0.6pt] table [x=x, y=profit_M3, col sep=comma] {Data/profit30M12310.csv};
		\addlegendentry{$M=3000$}
		\addplot[black!35!red, loosely dashed, line width=0.6pt] table [x=x, y=profit_M10, col sep=comma] {Data/profit30M12310.csv};
		\addlegendentry{$M=10000$}
		\end{axis}
		\end{tikzpicture}
    }
	\caption{\small{\em 
        The competition effects at $t = 30$. The first three plots show the distribution of miners' wealth.
        The last one shows the expected instantaneous profit for miners at different wealth levels, for four different competition levels.
        Parameters are the same as Figure \ref{fig:The distribution of miners' wealth.} except that $M$ takes multiple values. }}
	\label{fig: competition effect at 30.}
\end{figure}
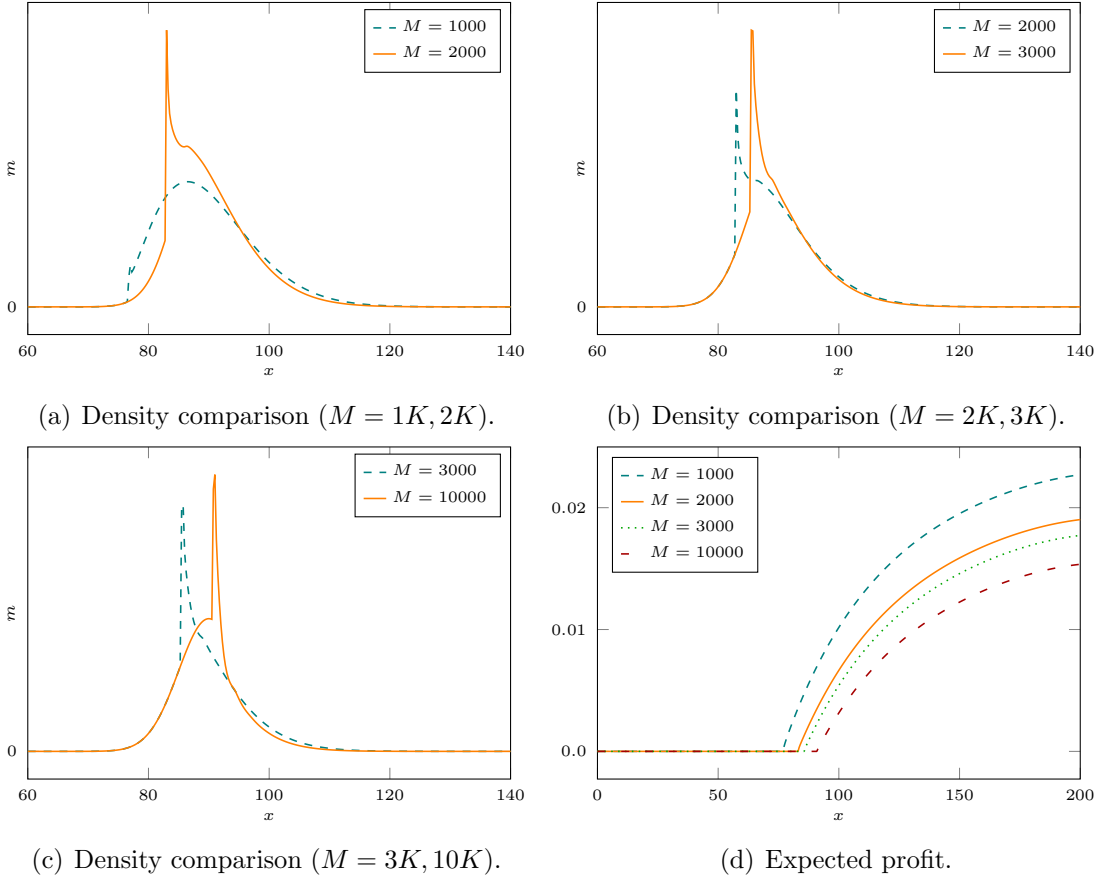
The competition parameter $M$ reduces the preferential attachment effect.
This is shown in Figures \ref{fig: competition effect at 30.} (a)(b)(c), as the density for $x \geq 100$ is lower for larger $M$.
Additionally, the
expected profit rate \eqref{instprofrate} decrease with respect to $M$, as is shown in Figure \ref{fig:profit30M12310}.
Meanwhile, as the competition becomes fierce and the participation threshold for miners to be active becomes larger.
When $M=1000$, miners need wealth around $75$ to have incentive to hash, but this increases to about $90$ for $M=10000$.
Hence, the competition makes the mining game less lucrative, and makes it harder for miners to stay active, which reduces the preferential attachment.

\section{Competition with cost advantages}\label{sec:cost_adv}
In this section, we consider a model in which one miner has cost advantages over the rest.
This could be due to access to cheaper energy or more advanced equipment,
and helps the miner become dominant in the mining game.
Bitmain is one example of an advantaged miner.
It utilizes cheaper electricity in China, like the hydropower stations in Sichuan during the rainy season,
and also Hydro Quebec in Canada, which offers some of the lowest electricity rates in North America.\footnote{\url{https://www.reuters.com/article/us-canada-bitcoin-china/chinese-bitcoin-miners-eye-sites-in-energy-rich-canada-idUSKBN1F10BU}}
The model studied in this section suggests that cost advantages can be a contributing factor in the centralization observed in Bitcoin mining, which is dominated by a few large entities, as is illustrated in Figure \ref{fig: Bitcoin hash rate distribution.}.

Because the advantaged miner is distinct from the other miners, 
 the general structure of Section~\ref{sec:model} is adapted in Section \ref{sec:cost_adv_problem}.
As in the earlier model, we are able to obtain the equilibrium explicitly when the smaller miners have exponential utility (Section 
\ref{sec:cost_adv_exp_liq}), and we can demonstrate wealth effects numerically when they have power utility (Section \ref{sec:cost_adv_crra}).

\begin{figure}[htbp]
	\centering
	\begin{tikzpicture}
	\small
	\pie[explode={0.2, 0.2, 0.2, 0.2, 0.2, 0.2, 0.2, 0.2, 0.2}]{20.7/BTC.com, 13.5/F2Pool, 12.4/AntPool, 9.7/Poolin, 9.1/BTC.TOP, 8.5/SlushPool, 7.4/ViaBTC, 5.0/BitFury, 13.7/Others}
	\end{tikzpicture}
	\caption{\small{\em Bitcoin hash rate distribution among the largest mining pools. The data is obtained on 06/30/2019 from \url{https://www.blockchain.com/pools}.}}
	\label{fig: Bitcoin hash rate distribution.}
\end{figure}
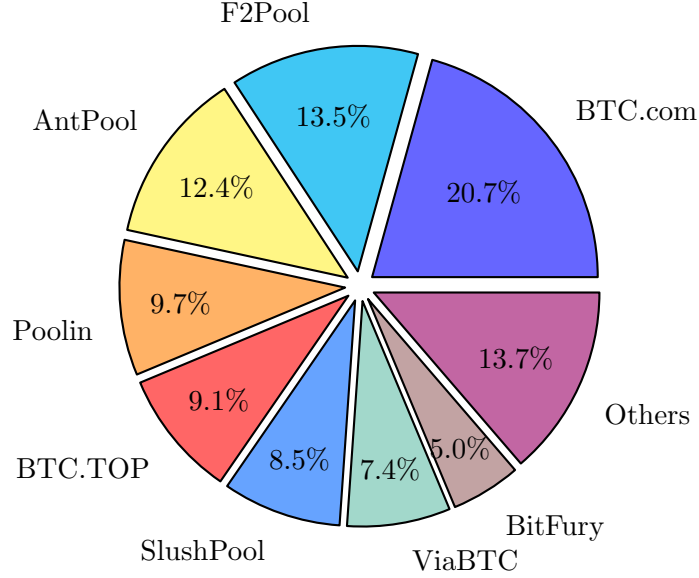

\subsection{Problem formulation}
\label{sec:cost_adv_problem}
We consider a cost-advantaged miner competing with $M+1$ individual miners
who are approximated, as in Section \ref{general_structure}, by a continuum of miners.\footnote{This type of competition between an individual and a continuum of payers is related to so-called major-minor mean field games, see e.g.\ \cite{huang_major_minor}. However, the introduction of our parameter $M$ to approximate an aggregate in terms of a mean implies that the so-called minor players are not really minor.}
This miner chooses a hash rate $\beta_{t}$,
and incurs cost $c_{1}\beta_{t} = k_c \costr{\beta_{t}}$ per hash,
where $0 < k_c \leq 1$ is the relative cost efficiency.
Given the (continuum) mean hash rate $\meanhashr_{t}$ of the individual miners,  the counting process $N^{1}_{t}$ with intensity
\[
    \rewardr^1_t = \frac{\K\beta_{t}}{(\beta_{t}+(M+1)\meanhashr_{t})}
\]
denotes the number of rewards received by the cost-advantaged miner.

We think of the advantaged miner as a profit-maximizing firm, or, more generally, having preferences described by a utility function $U_1(x)$ coupled with the appropriate set of admissible hash rate strategies.
The objective for the cost-advantaged miner is 
\begin{equation}
\label{advanced miner value function}
\sup_{\beta_{t} \geq 0}\E{\left[U_1\left(\int_{0}^{T}-c_{1}\beta_{t} dt+r\,dN^{1}_{t}\right)\right]}.
\end{equation}

As in Section \ref{general_structure}, the model for individual miners remains the same except that the intensity for $N_{t}$ in \eqref{wealth process for individual miner} becomes
\[
    \rewardr_t = \frac{\K\alpha_{t}}{(\alpha_{t}+M\meanhashr_{t}+\beta_t)},
\]
given the cost-advantaged miner's hash rate $\beta_{t}$.
Here the denominator consists of both the cost-advantaged miner's hash rate and the total of individual miners.

In this generality, we would need an additional state variable for the wealth of the cost-advantaged minor, as well as their value function, increasing the dimension and complexity of the PDE system. However, in the case where the players have exponential utility (or are risk-neutral), the equilibrium hash rates are wealth-independent, and we can solve explicitly (Section \ref{sec:cost_adv_exp_liq}). We can further handle numerically the case of the individual miners having power utility when the advantaged miner is profit-maximizing, i.e.\ risk-neutral, (Section \ref{sec:cost_adv_crra}), in which case we do not need an additional state variable.
In either of these cases then, any Markov control can be identified by a function $\beta(t; \meanhashr)$, i.e., the control only depends on time.

\subsubsection{HJB equation}\label{sec:cost_adv_individual}
The value function defined as in \eqref{base value function} now depends on both $\meanhashr$ and $\beta$, i.e.,
$v(t_{0}, x; \meanhashr, \beta)$.
For a fixed choice of $\meanhashr>0$ and $\beta\geq 0$, the HJB can be written as
\begin{equation}
\partial_t v + \sup_{\alpha\geq0}\left(-\costr{\alpha} \partial_x v + \frac{\K\alpha}{(\alpha+M\meanhashr_{t}+\beta_{t})}\Delta v\right) = 0,
\end{equation}
with terminal condition $v(T,x) = U(x)$.
Like Lemma \ref{property of the utility function base case}, it can be proved that $v$ is strictly increasing in $x$.
Hence, the maximizer is
\begin{equation}
\label{optimal individual hash rate in the advanced miner}
\alpha^{*}(t,x;\meanhashr, \beta)  = \left\{
\begin{aligned}
&-(M\meanhashr_{t}+\beta_{t})+\sqrt{\frac{(M\meanhashr_{t}+\beta_{t})\K\Delta v}{c\partial_x v}}, ~~~&&\text{if}~~~ M\meanhashr_{t} + \beta_{t} < \frac{\K\Delta v}{c\partial_x v}, \\
&0 &&\text{otherwise},
\end{aligned}
\right.
\end{equation}
and the HJB equation is
\begin{equation}
\left\{
\begin{aligned}
&\partial_t v + \left(\sqrt{c(M\meanhashr_{t}+\beta_{t})\partial_x v}-\sqrt{\K\Delta v}\right)^{2} = 0, ~~~&&\text{if}~~~ M\meanhashr_{t} + \beta_{t} < \frac{\K\Delta v}{c\partial_x v},\\
\label{HJB in the advanced miner case}
&\partial_t v = 0 &&\text{otherwise}.
\end{aligned}
\right.
\end{equation}

\subsubsection{Equilibrium characterization}\label{sec:cost_adv_equilibrium}
Let $m(t,x;\meanhashr, \beta)$ be the resulting density, corresponding to the optimal hash rate $\alpha^{*}(t,x;\meanhashr, \beta)$ of individual miners.
We say that $\eqmeanhashr$ and $\beta^{*}$ form an \emph{equilibrium} of the mining game with a cost-advantaged miner if 
\begin{equation}
\eqmeanhashr_t = \int_{\R}\opthashr(t,x; \eqmeanhashr,\beta^{*})m(t,x; \eqmeanhashr,\beta^{*})\,dx, ~~~\forall t \in [t_{0},T],
\end{equation}
and $\beta^{*}_{t} = \beta^{*}(t; \eqmeanhashr)$, from \eqref{advanced miner value function}.
Henceforth, let $\eqmeanhashr$ and $\beta^{*}$ denote the equilibrium mean hash rate and equilibrium hash rate for the cost-advantaged miner, $v(t, x) = v(t, x; \eqmeanhashr,\beta^{*})$, $\opthashr(t, x) = \opthashr(t, x; \eqmeanhashr,\beta^{*})$, and $m(t, x) = m(t, x; \eqmeanhashr,\beta^{*})$.

By the same argument presented in Section~\ref{sec:equilibrium}, it is meaningful to consider $\eqmeanhashr_t > 0$ for all $t$. 
Thus in the equilibrium,
we have coupled equations $\beta^{*}_{t} = \beta^{*}(t; \eqmeanhashr)$ and
\begin{equation}
\eqmeanhashr_{t} = -\eta(t)(M\eqmeanhashr_{t}+\beta^{*}_{t}) + \sqrt{(M\eqmeanhashr_{t}+\beta^{*}_{t})}\int_{E_t}\sqrt{\frac{\K\Delta v(t,x)}{c\partial_x v(t,x)}}m(t,x)\,dx,
\end{equation}
by integrating \eqref{optimal individual hash rate in the advanced miner} in $x$ over the set \eqref{level set for active miners}.
The Fokker-Planck equation is given by
\begin{equation}
\partial_t m - \partial_x (c\alpha^{*}(t,x)m) -\K\left(\frac{\alpha^{*}(t,x-r)}{\alpha^{*}(t,x-r)+M\eqmeanhashr_{t}+\beta^{*}_{t}}m(t,x-r)-\frac{\alpha^{*}(t,x)}{\alpha^{*}(t,x)+M\eqmeanhashr_{t}+\beta^{*}_{t}}m(t,x)\right) = 0,
\end{equation}
with initial distribution $m(t_{0},x) =m_{0}(x)$. 

\subsection{Exponential and risk-neutral utility miners: explicit solution}\label{sec:cost_adv_exp_liq}
Define
\begin{equation}
    \label{eqn:w}
    w(r;\gamma)=\left\{\begin{aligned}
    &\frac{(1-e^{-\gamma r})}{\gamma}\qquad & \gamma>0,\\
    & r & \gamma=0.
    \end{aligned}
    \right.
\end{equation}
When the higher cost miners have exponential utility ($\gamma > 0$) or are risk-neutral ($\gamma = 0$), and the cost-advantaged miner has exponential utility, or is risk-neutral ($U_1(x)=x$), we have the following result.
\begin{proposition}\label{prop:cost_adv_exp_liq}
    Suppose the individual miners have exponential utility $U(x) = -\frac{1}{\gamma}e^{-\gamma x}$ with $\gamma \geq 0$ and the cost-advantaged miner either i) has exponential utility with risk-aversion coefficient $\gamma_1>0$; or ii) is risk-neutral ($\gamma_1=0$).
    Suppose
	the relative cost efficiency satisfies
	\begin{equation}
	\label{cost efficiency of the advanced miner}
    k_{c} < \frac{w(r;\gamma_1)}{w(r;\gamma)} \frac{M+1}{M},
	\end{equation}
    and let
	\begin{equation}
	{\kappa} = \frac{\K w(r;\gamma)}{c}, \qquad \kappa_{1} = \frac{(M+1)\K w(r;\gamma_1)}{c_{1}}.
	\end{equation}
    Then, in equilibrium, all miners are active with
	\begin{equation}
	\label{equilibrium solution for the advanced miner and individuals}
	\alpha^{*}(t,x) \equiv \eqmeanhashr_{t} \equiv \frac{{\kappa}^{2}\kappa_{1}}{({\kappa}+\kappa_{1})^{2}} > 0, ~~~\beta^{*}_{t} \equiv \frac{{\kappa}\kappa_{1}(\kappa_{1}-M{\kappa})}{({\kappa}+\kappa_{1})^{2}} > 0,
	\end{equation}
	for all $t\in[t_{0},T]$ and $x\in \R$.
\end{proposition}

\begin{proof}
	With the ansatz $v(t,x) = U(x)h(t)$, the HJB \eqref{HJB in the advanced miner case} reduces to solving
	\begin{equation}
	\left\{
	\begin{aligned}
	&\partial_{t}h - \gamma \left(\sqrt{c(M\eqmeanhashr_{t}+\beta^{*}_{t})}-\sqrt{\K w(r;\gamma)}\right)^{2}h = 0, ~~~&&\text{if}~~~ M\eqmeanhashr_{t} + \beta^{*}_{t}< \frac{\K}{c}w(r;\gamma),\\
	&\partial_{t}h = 0 &&\text{otherwise},
	\end{aligned}
	\right.
	\end{equation}
	with terminal condition $h(T) = 1$.
	Since $\eqmeanhashr$ and $\beta^{*}$ are only functions of $t$, this validates the ansatz.
    In looking for an equilibrium in which $\alpha^*_t > 0$ and $\beta^*_t > 0$, we use the non-zero best response $\beta^*$ in \eqref{optimal lambda p}, and, using the ansatz in \eqref{optimal individual hash rate in the advanced miner}, we find
	\begin{equation} \label{advanced miner vs individual lambda}
	\alpha^{*}(t,x;\eqmeanhashr, \beta^{*})  = -(M\eqmeanhashr_{t}+\beta^{*}_{t})+\sqrt{ \frac{\K}{c}w(r;\gamma)(M\eqmeanhashr_{t}+\beta^{*}_{t})}.
	\end{equation}
	Since $\alpha^{*}$ does not depend on the wealth $x$, all individual miners are active.
	Therefore, we have $\eqmeanhashr_{t} = \alpha^*_t$.
	This, together with \eqref{optimal lambda p}, yields \eqref{equilibrium solution for the advanced miner and individuals}.
    It is direct that $\alpha^*$ is positive, and $\beta^*$ is positive if and only if \eqref{cost efficiency of the advanced miner} holds.
    Thus we have found the equilibrium in which everyone is active.
\end{proof}

\subsubsection*{Cost advantage and its effect on mining power concentration}
Proposition~\ref{prop:cost_adv_exp_liq} demonstrates that the cost-advantaged miner's efficiency leads to centralization in the following sense.
It can be checked that the hash rate $\beta^{*}_{t}$ in \eqref{equilibrium solution for the advanced miner and individuals} is increasing in $\kappa_{1}$ and hence decreasing in $c_{1}$.
Similarly, the hash rate $\alpha^{*}$ in \eqref{equilibrium solution for the advanced miner and individuals} of the individual miners increases with respect to $c_{1}$.
Thus, a smaller $c_{1}$---a larger cost advantage---makes the advantaged miner more dominant.
As a consequence, individual miners with higher cost have to decrease their hash rates to regulate their risk exposure, as the advantaged miner gets a larger share of the mining rewards.

To quantify this,we write 
$\rho = k_{c} \frac{w(r;\gamma)}{w(r;\gamma_1)}$. 
Then $\kappa_{1} = (M+1){\kappa}/\rho$. The probability for the cost-advantaged miner to get the reward is 
\begin{equation}
\frac{\beta^{*}}{\beta^{*} + (M+1)\eqmeanhashr} = \frac{\kappa_{1}-M{\kappa}}{{\kappa}+\kappa_{1}} = \frac{(1-\rho)M+1}{M+1+\rho} \approx 1-\rho
\end{equation}
for large $M$, whereas the remaining miners have a collective success probability $\rho$ and individual probability $\rho/(M+1)$.

If $\gamma \approx \gamma_1$, this simplifies further because $\rho \approx k_c$, i.e., the advantaged miner's share of the market is approximately their advantage $1 - k_c$.
If the cost-advantaged miner is $10\%$ efficient ($k_{c} = 0.9$), then $\rho \approx 0.9$, which gives a probability around $10\%$ for the cost-advantaged miner to get the reward.

This analysis shows that the advantaged miner receives an disproportionately large share of block rewards.
But because this miner pays a lower cost per hash, the profit increases even further.
Both of these effects remain in the power utility setting of the next section and are explored in more detail in Figure~\ref{fig:kc_effect} and the accompanying text.

\subsection{Power utility}\label{sec:cost_adv_crra}
We now consider wealth effects by endowing the individual miners with power utility preferences \eqref{mining power utility}.
As discussed in the final paragraph of Section~\ref{sec:cost_adv_problem}, we assume the cost-advantaged miner is profit-maximizing, which keeps the problem tractable without the need to introduce an additional state variable.
Therefore \eqref{advanced miner value function} becomes
\begin{equation}
\label{advanced miner value function1}
\sup_{\beta_{t} \geq 0}\E{\left[\int_{0}^{T}-c_{1}\beta_{t} dt+r\,dN^{1}_{t}\right]}.
\end{equation}
Given $\meanhashr>0$, the maximizer in \eqref{advanced miner value function1} satisfies the first-order condition
\begin{equation}
\label{foc for the pool}
-c_{1} + \frac{r\K(M+1)\meanhashr_{t}}{(\beta_{t}+(M+1)\meanhashr_{t})^{2}} = 0,
\end{equation}
which yields the best response
\begin{equation}
\label{optimal lambda p}
\beta^{*}(t; \meanhashr) = \left\{
\begin{aligned}
&-(M+1)\meanhashr_{t} + \sqrt{\frac{r\K(M+1)\meanhashr_{t}}{c_{1}}}, ~~~&&\text{if}~~~ \meanhashr_{t}<\frac{r\K}{c_{1}(M+1)}, \\
&0, &&\text{otherwise}.
\end{aligned}
\right.
\end{equation}
Like in Section~\ref{sec:illiquid}, we solve the model numerically.
The numerical method correspondingly follows the procedure in Section~\ref{sec:illiq_numerical} with appropriate changes adding $\beta$.

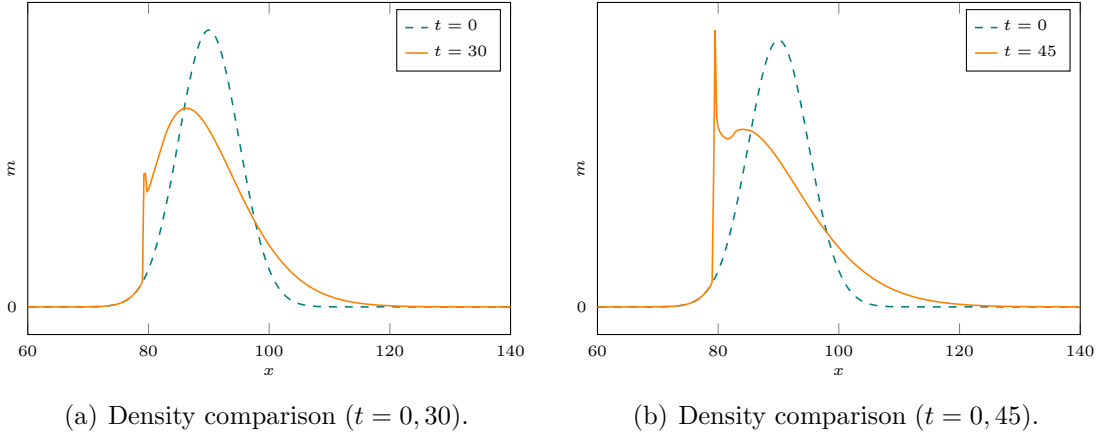
\begin{figure}
	\centering
    \newcommand{\axissettings}{enlarge x limits = 0}
    \newcommand{\legendsettings}{legend cell align={left}, legend image post style={scale=0.5}, legend style={nodes={scale=1.2}}}
		\begin{tikzpicture}
            \begin{axis}[width=0.85\textwidth,\axissettings,\legendsettings, xlabel=$x$, ylabel = $m$, ytick={0}, ylabel style={yshift=-0.5cm}, xlabel style={font=\tiny,at={(axis description cs:0.5,0.05)}},ylabel style={font=\tiny,at={(axis description cs:0.08,0.5)}}]
		\addplot[gray, line width=0.6pt, restrict x to domain=70:130] table [x=x, y=m0, col sep=comma] {Data/mkc80.csv};
		\addlegendentry{$t=0$}

        \addplot[orange, loosely dotted, line width=0.6pt, restrict x to domain=70:130] table [x=x, y=m30, col sep=comma] {Data/m.csv};
        \addlegendentry{$t=30$, without adv.\ miner}
        \addplot[orange, densely dotted, line width=0.6pt, restrict x to domain=70:130] table [x=x, y=m45, col sep=comma] {Data/m.csv};
        \addlegendentry{$t=45$, without adv.\ miner}

		\addplot[teal, densely dashed, line width=0.6pt, restrict x to domain=70:130] table [x=x, y=m30, col sep=comma] {Data/mkc80.csv};
		\addlegendentry{$t=30$, with adv.\ miner}
		\addplot[teal, loosely dashed, line width=0.6pt, restrict x to domain=70:130] table [x=x, y=m45, col sep=comma] {Data/mkc80.csv};
		\addlegendentry{$t=45$, with adv.\ miner}
		\end{axis}
		\end{tikzpicture}
	\caption{\small{\em 
        The distribution of miners' wealth at $t=0,30,45$, with $k_c = 0.8$ for the distributions with an advantaged miner.
        The parameters are the same as in Figure~\ref{fig:The distribution of miners' wealth.}.}}
        \label{fig:kc80_densities}
\end{figure}

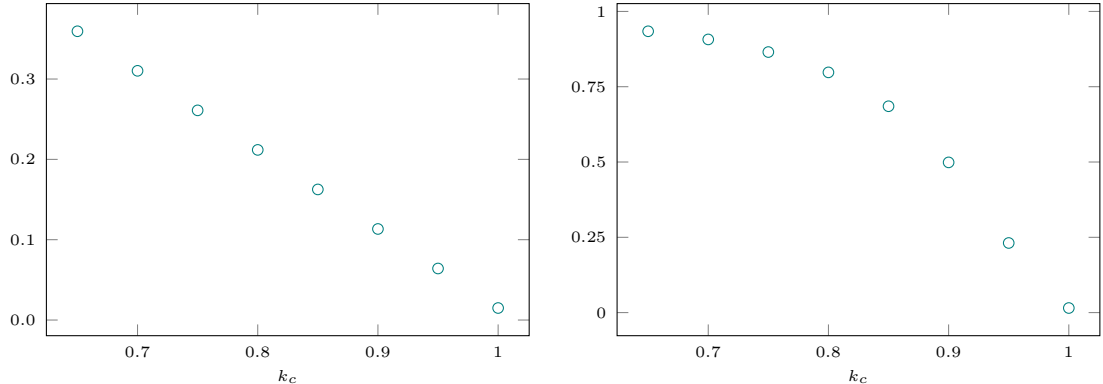
\begin{figure}
	\centering
    \newcommand{\axissettings}{width=8cm, height=6cm, enlarge x limits = 0}
    \newcommand{\legendsettings}{legend cell align={left}, legend image post style={scale=0.5}, legend style={nodes={scale=1.0}}}
	\subfigure[The advantaged miner's share of total rewards.]{\label{fig:kc_prob}
		\begin{tikzpicture}
            \begin{axis}[\axissettings,\legendsettings, enlarge x limits={abs=abs 1em}, 	
            y tick label style={
            	/pgf/number format/.cd,
            	precision=3,
            },
            yticklabel=\phantom{0}$\mathllap{\tick}$, xlabel=$k_c$, ylabel style={yshift=-0.5cm}, xlabel style={font=\tiny,at={(axis description cs:0.5,0.05)}},ylabel style={font=\tiny,at={(axis description cs:0.08,0.5)}}]
		\addplot[teal, only marks, mark=o] table [x=kc, y=prob, col sep=comma] {Data/prob30kc.csv};
		\end{axis}
		\end{tikzpicture}
    }
    \subfigure[The advantaged miner's share of total profits.]{\label{fig:kc_profit}
		\begin{tikzpicture}
            \begin{axis}[\axissettings,\legendsettings, enlarge x limits={abs=abs 1em}, 	
            y tick label style={
            	/pgf/number format/.cd,
            	precision=3,
            },
            ytick={0,0.25,0.5,0.75,1.0},
            xlabel=$k_c$, ylabel style={yshift=-0.5cm}, xlabel style={font=\tiny,at={(axis description cs:0.5,0.05)}},ylabel style={font=\tiny,at={(axis description cs:0.08,0.5)}}]
		\addplot[teal, only marks, mark=o] table [x=kc, y=profit_ratio, col sep=comma] {Data/profit_ratio30kc.csv};
		\end{axis}
		\end{tikzpicture}
    }
	\caption{\small{\em 
        Effect of cost efficiency, $k_c$.
        The first shows the advantaged miner's probability of getting the next reward, i.e., the expected share of total rewards.
        The last one shows the advantaged miner's expected instantaneous profits divided by the total instantaneous profits of all miners.
        For both figures, $t=30$.
        The parameters are the same as in Figure~\ref{fig:The distribution of miners' wealth.}.}}
        \label{fig:kc_effect}
\end{figure}

Figure \ref{fig:kc80_densities} shows the wealth distribution of individual miners when $k_{c} = 0.8$ for $t=0,30,45$.
Two time points from Figure~\ref{fig:The distribution of miners' wealth.} are also plotted, which show two subtle differences.
Recall from Section~\ref{sec:concentration_of_wealth} and Figure~\ref{fig:The distribution of miners' wealth.} that the left peak of each curve is the accumulation of miners at the point where mining is no longer optimal.
Most clearly visible \label{bit} is the right-shift of this peak, in the presence of an advantaged miner.
This means that the minimum wealth necessary for non-zero mining to be incentivized is greater with an advantaged miner, thus increasing centralization.
The other effect is the slightly slower dispersion of the distribution, which is explained by the lower mining rate across the board as a result of hesitancy due to the added competition from the advantaged miner.
Nevertheless, the type of evolution in Figure~\ref{fig:The distribution of miners' wealth.} is observable also in Figure~\ref{fig:kc80_densities}.

The effect of varying the cost efficiency $k_c$ is plotted in Figure~\ref{fig:kc_effect}.
Figure \ref{fig:kc_prob} shows the advantaged miner's probability of getting the next reward, which is also the share of the expected instantaneous reward.
Despite moving to power utility mining, the relationship is almost identical to that found analytically with exponential utility in Section~\ref{sec:cost_adv_exp_liq}:
The hash rate share taken by the advantaged miner is approximately $1 - k_c$.
We see that as the cost efficiency moves from $k_{c} = 1.0$ to $0.65$, the advantaged miner's hash rate increases from around $1\%$ to $35\%$.
This is a strong effect, and hence the cost advantages could be one explanatory factor for the concentration of mining power.
A similar idea also appears in \citep{Weinberg2018}.
They suggest that if a miner's cost is (e.g.) 10\% lower than those of other miners, then the miner must control at least 10\% of the total mining power.
\cite{Capponi2019} argue that miners invest in R\&D which allows them to develop more energy efficient mining equipment. Hence miners can have lower marginal cost and contribute higher hash rates.
As $M=1000$, the advantaged miner contributes a somewhat higher hash rate than the rest of the population---which is on the order of $1/M$---even for $k_c = 1$.
We attribute this small difference to the risk aversion, as these numbers are for $\gamma = 0.8$ and $\gamma_1 = 0$.

Figure~\ref{fig:kc_profit} plots the share of total profits for the advantaged miner.
At a $35\%$ cost advantage, i.e., $k_c = 0.65$, the advantaged miner reaps $93\%$ of the total profits generated, and $86\%$ of profits for $k_c = 0.75$.
This shows that most of the economic welfare in the system is received by a miner with a cost advantage.

\section{Conclusion and further research}\label{sec:conclusions}

This paper develops stochastic models for the mining of cryptocurrencies that implement a proof-of-work protocol.
As miners compete through their computation efforts---their hash rates---a rational utility maximization framework is analyzed as a novel mean field game of intensity control.
Remarkably, the equilibrium is found explicitly for certain utility functions---one of the very few explicit solutions of mean field games outside the linear quadratic framework.
In other cases, the equilibrium can be found numerically and efficiently by methods described in this paper.

Our main finding is that, as a result of initial wealth heterogeneity among the miners, more rewards tend to be collected by those who have more wealth, while miners with lesser wealth may be disincentivized from participating in mining activity at all.
This leads to a feedback effect of greater wealth and more concentrated mining. Concentration goes against the very fundamental principle of mining being decentralized, which is a requisite for the security of the cryptocurrency.

Incorporating a player with even a slight cost-advantage into our game shows that they become relatively dominant. This is consistent with how these professional miners with more advanced equipment or access to cheaper electricity have come to account for a significant share of mining power in recent years.

There are plenty of directions for future research that can build on the approach to cryptocurrency mining presented in this paper. We present some preliminary findings on three of these directions in the appendices, as we now describe.

At the end of 2010, the first mining pool was announced. Nowadays, most computational power comes from mining pools that are controlled by a few companies (see Figure \ref{fig: Bitcoin hash rate distribution.}). For instance, AntPool and BTC.com are run by Bitmain. Miners can join a pool, which collects their computational power to do the mining. Once the pool gets a reward, it is shared among the miners. 
Meanwhile, these companies also contribute a significant proportion of hash rates to their own pools.
Although pools are often public, meaning miners may participate and receive a share of pool revenue, the pool operators control what block data (the `work') is processed by its participants.
This means that a few entities account for selecting the block data (transactions) processed by a large share of hash rates in the world.
An initial discussion of how our model can be adapted to incorporate mining pools is presented in Appendix~\ref{pools}.

A notorious feature of cryptocurrencies, in particular bitcoin, is wild changes in their prices, which can run to speculative highs followed by sudden crashes. 
In Appendix~\ref{sec:stoch_price}, we explore the effect of stochastic rewards, for instance due to price fluctuations, in an extension of our model. 
In the case of exponential utility miners, we find miners react to the level but not the volatility of prices. 

We have assumed homogeneous costs and preferences in our model, which may be relaxed as we demonstrate (for costs) in Appendix~\ref{sec:cost_heterogeneity}. In general, costs would depend on miners' geographic location, access to silicon and hardware, etc.
Bill Tai of Hut 8 Mining Corp.\ has stated that big miners (like Bitfury) can buy at discount thanks to being able to ``buy silicon in large quantities and commit to the electricity grid in chunk sizes.''%
\footnote{\url{https://www.bloomberg.com/news/articles/2018-04-18/bitcoin-miners-facing-a-shakeout-as-profitability-becomes-harder}}
This suggests that per unit costs should be smaller for large miners.
Such concave costs would amplify the preferential attachment effects we find already in the linear costs model.
Another avenue for further research is to include fixed costs of entering the mining game, as motivated in \cite{garratt2020fixed}.

\begin{appendices}
\section{Proofs}\label{sec:proofs}
\subsection{Proof of Lemma \ref{property of the utility function base case}\label{proof2.1}}

\begin{proof}
	For the finiteness, by Jensen's inequality,
	\begin{equation}
	\mathbb{E}[U(X_{T})|X_{t} = x] \leq U(\mathbb{E}[X_{T}|X_{t} = x]) \leq U(x + r\mathbb{E}[N_{T} - N_{t} | X_{t} = x]) \leq U\left(x + r\K T\right) < \infty.
	\end{equation}
	
	For any $x_{1} < x_{2}$, let $\alpha^{i}_{t}$ and $X_{t}^{i}$ ~$(i =1,2)$ denote the optimal hash rate and corresponding wealth starting at time $t$ with initial wealth $x_{i}$.
	Consider the case where we start with $x_{2}$.
	We use $x_{1}$ as the wealth in the mining and save $x_{2}-x_{1}$.
	Then we use the hash rate $\alpha^{1}_{t}$. The corresponding wealth process is denoted by $X_{t}^{2,1}$. Thus we have $X_{t}^{2,1} \geq x_{2} - x_{1} +  X_{t}^{1} > X_{t}^{1}$. Thus,
	\begin{equation}
	v(t, x_{1}) < \mathbb{E}[U(X_{T}^{2,1})|X_{t}^{2,1} = x_{2}] \leq v(t, x_{2}).
	\end{equation}
\end{proof}

\subsection{Proof of Proposition \ref{cutoff of the wealth level}}
Since $\alpha$ can be chosen identically zero on $[t, T]$ and the optimal strategy does at least as well as doing nothing,
for any $\varepsilon > 0$, we have
\[
    \frac{v(t, \varepsilon) - v(t, 0)}{\varepsilon} \geq \frac{U(\varepsilon) - U(0)}{\varepsilon} \xrightarrow{\varepsilon \to 0} \infty,
\]
from the derivative of the power utility function \eqref{mining power utility}.
Hence, $\partial_x v(t, x) \to \infty$ as $x \to 0$ and
consequently, $\Delta v / \partial_x v$ is arbitrarily small in some neighborhood of $0$, because $\Delta v(t, 0) \leq v(t, r) - v(t, 0) = v(t, r) < \infty$.
Thus, by \eqref{base optimal hash rate}, there exists a $x_{\mathrm{off}}(t)$ such that zero rate mining is optimal for $x \leq x_{\mathrm{off}}$.

\section{Mean field games (of controls)}%
\label{app:mfg}
We here describe the ideas behind mean field games and the arguments leading to the continuum limit.
Most mean field games models exhibit interaction between the players through their state.
In contrast, the reward rate to each player in the proof-of-work mining game depends only on the hash rate of the population.
Consequently, the interaction between players is through their \emph{actions} as opposed to their state.
This is commonly referred to as \emph{extended mean field games} or \emph{mean field games of controls}.
Conceptually, the state and action interaction types of mean field games are very similar, but the structure of the resulting equations are very different.

Consider a game of $N$ players, each ($i \in 1,\dots,N$) with control $\alpha^i$ and state $X^i = X^{i, \alpha^i}$ described by the evolution
\begin{equation}
    \label{eqn:mfg_dynamics}
    \dif X^i_t = f\Bigl(\alpha^i_t, X^i_t, \smash{\sum_{j=1}^N} \alpha^j_t / N\Bigr) \dif t + \dif W^i_t, \quad X^i_0 = x^i,
\end{equation}
where $W^i$ is a Brownian motion representing player $i$'s idiosyncratic noise ($W^i$ and $W^j$ are independent for $i\neq j$).
Given a strategy profile $\alpha^{-i} = (\alpha^1, \dots, \alpha^{i-1}, \alpha^{i+1}, \dots, \alpha^N)$ of the other players,
player $i$ is optimizing some quantity $\E[U(X^{i, \alpha^i}_T)]$, and we assume that the optimizer $\alpha^{i,*}$ can be found in the class of Markovian controls, i.e., that it is a function of player $i$'s state (for fixed $\alpha^{-i}$).
In other words, we may write $\alpha^{i,*}_t = \alpha^*(t, X^i_t, \alpha^{-i}_t)$.

An equilibrium $t \mapsto \alpha_t = (\alpha^*(t, X^1_t, \alpha^{-1}_t), \dots, \alpha^*(t, X^N_t, \alpha^{-N}_t))$ of this form is characterized by the property that for all $i$, $t \mapsto \alpha^*(t, X^i_t, \alpha^{-i}_t)$ is an optimizer to player $i$'s optimization problem, i.e., no player has anything to gain from deviating.
We will not delve into detail about why these games become very difficult to solve, but the gist of it is that each player must solve an HJB equation that depends on every other player's solution, thus creating a system of $N$ coupled equations.
The idea of mean field games enters to circumvent this dimensionality issue.

The first step is to consider a sequence of games with increasingly many players.
Each additional player must be assigned an initial state (see state $x^i$ for player $i$ in \eqref{eqn:mfg_dynamics}).
This is done by sampling the initial states from a distribution (independent from $W^i$) with density $m_0$.
We now consider what happens to the problem in the limit $N \to \infty$.
Because the probability of two players starting at the same state is zero, we may index each player by their state $x^i$ instead of $i$.
In the limiting problem we simply write $x$ and remember that the collection of initial states $x$ is distributed according to $m_0$.
As time passes, the players' states evolve.
We denote by $m(t, \cdot)$ the density of players at time $t$.\footnote{This density does depend on the actions, but for now we omit this in the notation.}

For any strategy profile $\alpha = (\alpha^{x^1}, \dots, \alpha^{x^N})$, we consider the problem of player $x^i$.
The first observation is that if player $x^i$ deviates from $\alpha^{x^i}$, the impact on the other players is of order $1/N$, due to the averaging effect in \eqref{eqn:mfg_dynamics}.
In particular, as $N \to \infty$, we see that the effect of any one player on the others is vanishing, as
\[
    f\Bigl(\alpha^{x^i}_t, X^{x^i}_t, \smash{\sum_{j=1}^N} \alpha^{x^j}_t / N\Bigr)
    \longrightarrow
    f\Bigl(\alpha^{x^i}_t, X^{x^i}, \int_{\R} \alpha^{x^j}_t m(t,x^j) \dif x^j \Bigr).
\]
Hence, in the limit, player $x^i$'s action and state alone do not impact the population average.
Moreover, the dependence of player $x^i$ on the population is only through a statistical property: the mean.
Dropping the superscript, we may thus consider the dynamics of player $x$ given the mean control $\bar{\alpha}$ of the population:
\begin{equation}
    \label{eqn:mfg_continuum_dynamics}
    \dif X^{x, \alpha^x} = \dif X^x_t = f(\alpha^x_t, X^x_t, \bar{\alpha}_t) \dif t + \dif W^x_t, \quad X^x_0 = x.
\end{equation}
Player $x$ seeks to optimize $\E[U(X^x_T)]$, given the (mean) strategy profile $\bar{\alpha}$.
For any fixed $\bar{\alpha}$, this is a standard control problem, and we again write the optimizer as a function of the current state and the actions of the other players: $\alpha^*(t, X^x_t, \bar{\alpha}_t)$.

If every player chooses the action $\alpha^*_t = \alpha^*(t, X^x_t, \bar{\alpha}_t)$, we denote by $m(t, \cdot; \alpha^*) = m(t, \cdot; \alpha^*, m_0)$ the resulting density at time $t$.
The process $\bar{\alpha}$ is the equilibrium mean control if
\begin{equation}
    \label{eqn:mfg_equilibrium_cond}
    \bar{\alpha}_t = \int_{\R} \alpha^*(t, x, \bar{\alpha}_t) m(t, x; \alpha^*) \dif x.
\end{equation}
In other words, if all other players are using the strategy $\alpha^*$, then no players can improve their situation by deviating from this strategy.

Finally, given any process $\bar{\alpha}$ and a strategy $\alpha(t, x)$ shared by all players, the density $m(t, \cdot; \alpha)$ of the population can be shown to satisfy the Fokker--Planck equation
\begin{equation}
    \label{eqn:mfg_fokker-planck}
    \partial_t m(t, x; \alpha) - \partial_x (f(t, \alpha(t, x), x, \bar{\alpha}) m(t, x; \alpha)) - \frac{1}{2} \partial_{xx} m(t, x; \alpha) = 0, \quad m(0, \cdot, \alpha) = m_0.
\end{equation}
Furthermore, the value function $v$ of each player's optimization problem is characterized by the HJB equation
\begin{equation}
    \label{eqn:mfg_HJB}
    \partial_t v + \sup_{\alpha} f(t, \alpha, x, \bar{\alpha}) \partial_x v + \frac{1}{2} \partial_{xx} v = 0, \quad v(T, \cdot) = U.
\end{equation}
To summarize:
by considering the limit as $N \to \infty$, the system of $N$ coupled HJB equations reduces to the two coupled equations \eqref{eqn:mfg_fokker-planck} and \eqref{eqn:mfg_HJB} along with the equilibrium fixed point condition \eqref{eqn:mfg_equilibrium_cond}.
Although the structure of this smaller system is still mathematically very complex, it is often amenable to numerical computations, in contrast to the computational intractability of the $N$-player system.

\section{Stochastic reward \& exponential utility}
\label{sec:stoch_price}
We consider the base model of Section \ref{sec:model}, but when the reward $r$ is stochastically varying over time, driven, for instance, by the price of bitcoin. Similar to Section \ref{sec:liquid}, we can solve analytically for the equilibrium when the miners have exponential utility.
A miner's wealth process $X=(X_{t})_{t \geq t_0}$ follows:  
$dX_{t} = - \costr{\alpha_{t}}\,dt + r_t\,dN_{t}$, 
where now $r>0$ is modeled as a Markov process with infinitesimal generator $\lop_r$, whose increments are independent of the contemporaneous increments $dN$.
For instance, if $r$ is fit to a geometric Brownian motion model:
$dr_t= \mu r_t \,dt + \sigma r_t \,dW_t$, 
 where $W$ is an independent Brownian motion and $\sigma$ is the reward volatility, then
 $\lop_r = \frac12c_1^2r^2\partial_{rr} + c_0r\partial_r$.
 
Consider exponential utility miners, i.e., $U(x):=-\frac{1}{\gamma}e^{-\gamma x}$.
Then the value function of an individual miner, which is a function of $t$, $x$, and $r$, is defined by 
\begin{equation}
\label{base value function1}
v(t,x,r; \meanhashr) = \sup_{\alpha\geq0}\E[U(X_{T})|X_{t} = x, r_t=r].
\end{equation}
The associated HJB equation is
\begin{equation}
\label{base HJB1}
\partial_t v + \lop_r v+\sup_{\alpha\geq0}\left(-\costr{\alpha} \partial_x v + \frac{\K\alpha}{(\alpha+M\meanhashr_{t}(r))}\Delta v\right) = 0,
\end{equation}
with terminal condition $v(T,x,r) = U(x)$, and where 
$\Delta v= v(t,x+r,r;\meanhashr) - v(t,x,r;\meanhashr)$. 

Adapting the analysis of Section \ref{oneminerprob},
in equilibrium the optimal hash rate $\opthashr$ and $\eqmeanhashr$ are related by
\begin{equation}
    \opthashr(t, x, r; \eqmeanhashr) =
    -M\eqmeanhashr_t(r) + \sqrt{\frac{\K M \eqmeanhashr_{t}(r)\Delta v(t,x,r;\eqmeanhashr)}{c\partial_x v(t,x,r;\eqmeanhashr)}},
\label{base optimal hash rate1}
\end{equation}
and the HJB equation can be written as:
\begin{equation}
\partial_t v + \lop_r v + \left(\sqrt{Mc\eqmeanhashr_t(r) \partial_x v}-\sqrt{\K\Delta v}\right)^{2} = 0.
\label{base HJB after optimization1}
\end{equation}

With exponential utility miners, we have $v(T,x,r) = U(x):=-\frac{1}{\gamma}e^{-\gamma x}$. 
With the ansatz $v(t,x,r)=g(t,r)U(x)$, we find that $g$ must solve the linear PDE problem
 $$ \partial_{t}g + \lop_r g- \gamma \left(\sqrt{Mc\eqmeanhashr_t(r)}-\sqrt{\K w(r;\gamma)}\right)^{2}g = 0, \qquad g(T,r)=1, $$
where $w$ was defined in \eqref{wdef}.
The ansatz implies that $\Delta v$/$\partial_x v = w(r; \gamma)$ does not depend on $x$.
Then, as in the proof of Proposition \ref{base case exponential utility},
all miners are active and so
	$$\eqmeanhashr_t(r) = \frac{\K M}{c(1+ M)^{2}}w(r;\gamma). $$
This yields the linear PDE problem for $g$:
		\begin{equation}
	\partial_{t}g +\lop_r g - \gamma\frac{\K w(r;\gamma)}{(1+M)^{2}}\,g = 0,
	\end{equation}
	with terminal condition $g(T,r) = 1$.
The solution can be written as a Feynman--Kac expectation:	
	\[ g(t,r) = \E\left[\exp\left(-\gamma\int_t^T\frac{\K w(r_s;\gamma)}{(1+M)^{2}}\,ds\right) \mid r_t=r\right],
	\]
	but is not needed for the optimal mining rate, which is given by
\begin{equation}
	\label{exponential optimal rate1}
	\alpha^{*}(t, x, r; \eqmeanhashr) \equiv \eqmeanhashr_t(r) \equiv \frac{\K M}{c(1+M)^{2}} w(r;\gamma).
	\end{equation}
	The individual reward rate is
$\rewardr_t \equiv \frac{1}{D(1+M)}$,
which does not depend on $r$.
So, even as the price appreciates or falls, the hash rates go up and down respectively, following the formula \eqref{exponential optimal rate1}, but the individual rewards rates remain the same.
The reward affects the value function as higher reward gives higher reward, but reward volatility $\sigma$ 
does not enter into the equilibrium hash rates \eqref{exponential optimal rate1}, which are simply reactive to the present reward.
	
\section{Mining pools\label{pools}}
An important property of mining pools is their use for risk sharing. 
Miners may join pools to reduce risk by sharing rewards among all pool contributors, in some proportion to their contribution.
Each miner thus receives smaller but more frequent rewards,\footnote{%
  Some pools even employ a system in which members are paid for hash contributions regardless of reward arrival, thus moving all risk to the pool.
  Such pools charge higher fees to offset this risk.
  As mentioned towards the end of this section, such a structure fits well into the analysis in the main text.
}
thereby reducing the risk associated with reward arrival.
In fact, a miner may join multiple pools and contribute hashes to the different pools.
This enables further risk reduction.

However, pool miners no longer have the power to select which transactions to process as part of their blocks.\footnote{The decentralized mining pool P2Pool is an example of an exception to this.}
Instead the block data---the `work'---is handed down to them by pool operators.
So even if the risk sharing mechanism of mining pools improves the resilience of the blockchain to outside attacks, it does come at a cost to decentralization, and thus lesser resilience to malicious behavior among its participants.

We here consider a model with pools and miners who desire the risk sharing offered by pools.
The decision making problem of pools becomes a problem very similar to that studied in the paper, suggesting that similar competition and centralization effects would appear on a pool level.

There exists a variety of reward distribution schemes, fee structures, and other features that differentiate pools.
To not overcomplicate the model, we consider pools that distribute rewards to their contributors proportional to their hash rate contribution.\footnote{%
  More commonly, pools use a `Pay Per Last N Shares' (PPLNS) system that discourages contributing to multiple pools.
  Such a system would see more centralization than the model set up here.
}
We also assume that all pools are public, i.e., anybody may join them, and that all miners who are not operating a pool distribute their contributions across pools to reduce risk.
There are thus two groups of miners: pool operators and individual miners.

The assumptions that all pools are public and that everybody mines for pools allows individual miners perfect risk sharing, except for the inherent randomness of the arrival of the next block.
This is accomplished by mining for each pool in proportion to the pool's total contribution to the global hash rate.
We denote by $\alpha^{\text{ind}}$ the aggregate hash rate of the individual miners, and  
by $\alpha^p$ the hash rate of the operator of pool $p = 1,\dots,P$.
The hash rate contribution of individual miners to pool $p$ is then
\[
  \alpha^{\text{ind}} \frac{\alpha^p}{\sum_{p'=1}^P \alpha^{p'}}.
\]
Given the individual miners' contributions to the pool, each pool's hash rate is
\[
  \alpha^p + \alpha^{\text{ind}}\frac{\alpha^p}{\sum_{p'=1}^P \alpha^{p'}} = \alpha^p \Big(1 + \alpha^{\text{ind}}\frac{1}{\sum_{p'=1}^P \alpha^{p'}}\Big),
\]
and the aggregate is
\[
  \sum_{p=1}^P \alpha^p \Big(1 + \alpha^{\text{ind}}\frac{1}{\sum_{p'=1}^P \alpha^{p'}}\Big).
\]
A pool's block arrival rate is (proportional to) the ratio of the two which, after simplification, is
\[
  \frac{\alpha^p \Big(1 + \alpha^{\text{ind}}\frac{1}{\sum_{p'=1}^P \alpha^{p'}}\Big)}{\sum_{p'=1}^P \alpha^{p'} (1 + \alpha^{\text{ind}}\frac{1}{\sum_{p'=1}^P \alpha^{p'}})}
  = \frac{\alpha^p}{\sum_{p'=1}^P \alpha^{p'}} = \frac{\alpha^p}{\alpha^p + \sum_{\substack{p'=1\\p'\neq p}}^P \alpha^{p'}}.
\]
Hence, when individual miners strive for risk sharing, the reward arrival for pools has the same structure as without pools.

Suppose a pool charges a fee $\psi r$ when a block is mined.
The pool's reward is then
\begin{align*}
  r (\text{share in pool}) + r \psi (\text{share of participants})
  &=r \frac{\alpha^{p}}{\alpha^p + \alpha^{\text{ind}} \frac{\alpha^p}{\sum_{p'=1}^P \alpha^{p'}}}
  + r \psi \frac{\alpha^{\text{ind}} \frac{\alpha^p}{\sum_{p'=1}^P \alpha^{p'}}}{\alpha^p + \alpha^{\text{ind}} \frac{\alpha^p}{\sum_{p'=1}^P \alpha^{p'}}} \\
  &= r \Big(1 + (\psi - 1) \frac{\alpha^{\text{ind}}}{\sum_{p'=1}^P \alpha^{p'} + \alpha^{\text{ind}}}\Big).
\end{align*}
This reflects that the pool operator must share the reward with participants in proportion to the relative hash contributions.

We thus have a model for the pool manager's wealth as
\[
  \dif X^p_t = - c \alpha^p_t \dif t + \tilde{r}^p_t \dif N^p_t,
\]
where $N^p_t$ has intensity
\[
  \rewardr^p = \frac{\alpha_t^p}{\alpha_t^p + \sum_{\substack{p'=1\\p'\neq p}}^P \alpha_t^{p'}}
  \approx
  \frac{\alpha_t^p}{\alpha_t^p + (P - 1) \bar{\alpha}_t}
\]
and
\[
  \tilde{r}_t^p = r \Big(1 + (\psi - 1) \frac{\alpha_t^{\text{ind}}}{\sum_{p'=1}^P \alpha_t^{p'} + \alpha_t^{\text{ind}}}\Big)
  \approx r \Big(1 + (\psi - 1) \frac{\alpha_t^{\text{ind}}}{\alpha_t^p + (P - 1) \bar{\alpha}_t + \alpha_t^{\text{ind}}}\Big),
\]
where $\bar{\alpha}$ is the continuum approximation of the average pool manager hash rate.

The continuum mean field game version is to consider a representative pool manager with initial wealth $X_0 = x$ and hash rate $\alpha_t$, and wealth evolving according to
\begin{equation}\label{eqn:pool_mfg}
  \dif X_t = - c \alpha_t \dif t + \tilde{r}_t \dif N_t,
  \quad
  \rewardr_t = \frac{\alpha_t}{\alpha_t + (P - 1) \bar{\alpha}_t},
  \quad
  \tilde{r}_t = r \Big(1 + (\psi - 1) \frac{\alpha_t^{\text{ind}}}{\alpha_t + (P - 1) \bar{\alpha}_t + \alpha_t^{\text{ind}}}\Big).
\end{equation}
The pool game \eqref{eqn:pool_mfg} differs slightly from our model in Section~\ref{general_structure}, where $r$ is constant.
However, because in the expression for $\tilde{r}$, $\alpha$ only appears in the denominator, along with both $\alpha^{\text{ind}}$ and all other pool operators' hash rates, the representative operator has very little control over the reward size.
Thus, the choice of $\alpha$ primarily impacts the reward arrival rate $\rewardr$, where it appears also in the numerator.
In other words, this model is very similar to that studied in this paper and can be treated in the same way.

The similarity of this model to that without pools suggests the competition and centralization would be similar among pool operators after taking pools into account.
In particular, the previous results suggest the emergence of dominant pools, which we also observe in practice.
This section assumes a proportional distribution of rewards to participants.
If the pools instead employ a pay per share (PPS) system, which means miners receive a guaranteed payment per hash, pools essentially buy hashes at a cost equal to that payment.
This cost can be interpreted as $c$ in our model, which leads to a game even more similar to that in the main text.

Of course, a more complete model of pools might include competition in setting the fee $\psi$.
For further discussion and analysis of pooled mining in a static setting, we refer to \cite{NBERw25592}.

\section{Cost heterogeneity}
\label{sec:cost_heterogeneity}
Recall that in the model of Section~\ref{sec:liquid}, when the individual miners have CARA utility
$\Delta v / \partial_x v$ is constant.
Consequently, $\alpha^*$ in \eqref{base optimal hash rate} becomes
\begin{equation}
\label{alphastar}
\alpha^{*} = \left\{
\begin{aligned}
&-M\meanhashr + \sqrt{\frac{w(r;\gamma)\K M\meanhashr}{c}}, ~~~&&\text{if}~~~ \meanhashr<\frac{w(r;\gamma)\K}{cM}, \\
&0, &&\text{otherwise},
\end{aligned}
\right.
\end{equation}
with $w$ defined in \eqref{wdef},
which is independent of wealth.

Suppose now that the continuum of miners have heterogeneous costs.
In Section~\ref{oneminerprob} miners were indexed by their wealth $x$, whereas here the miners are characterized by their wealth $x$ and the parameter $c$ indexing their cost.
Because of the wealth-independence of the optimal hash rates described above, we may consider only $c$ when indexing the miners.
For simplicity, consider costs $c$ in some bounded interval $(c_0, c_1) \subset \R_+$.
Denote by $f_c(c)$ the associated density of miners.
Then, the individual miners hash at rates \eqref{alphastar}.
We use the notation $\alpha^*(c)$ to stress the dependence on $c$.

We see from \eqref{alphastar} that
a miner is active if its cost
$$ c<\frac{w(r;\gamma)\K}{\meanhashr M}.$$
Define $c_{\max} = \min\{ \frac{w(r;\gamma)\K}{\meanhashr M}, c_1 \}$.
By this definition and \eqref{alphastar}, miners with costs $c < c_{\max}$ are active, while the rest do not mine.
Then, averaging \eqref{alphastar} over active miners, i.e., over $c < c_{\max}$ with respect to the density $f_c(c)$, leads to the following equation for $\meanhashr$:
\begin{equation}
  \label{eqn:alphabarz}
\meanhashr = \int_{c<c_{\max}}\left(-M\meanhashr + \sqrt{\frac{w(r;\gamma)\K M\meanhashr}{c}}\right)f_c(c)\,dc.
\end{equation}

In a case where all miners are active, i.e., $c_{\max} = c_1$, we can solve \eqref{eqn:alphabarz} to find
$$ \meanhashr = \frac{w(r;\gamma)\K M}{\bar c(M+1)^2}. $$
where $\bar c$ is the inverse square root averaged cost
$$ \frac{1}{\sqrt{\bar c}} = \int_{c_0}^{c_1} \frac{1}{\sqrt{c}}f_c(c)\,dc. 
$$
Again using \eqref{alphastar}, the condition for a miner to be active is therefore
$$ c < \left(\frac{M+1}{M}\right)^2\bar c,$$
and so costs have to be increasingly homogeneous for larger $M$ for everyone to be active.
The cost distribution may have a small tail to the left, but on the right miners would not participate.
This suggests that homogeneous costs, possibly with few cost-advantaged miners, is the right place to start, as we did in this paper.

\end{appendices}

	\bibliographystyle{apacite}
\small
\bibliography{paper}

\end{document}